\documentclass[12pt,english]{article}

\usepackage{mathptmx}
\DeclareMathAlphabet{\mathcal}{OMS}{cmsy}{m}{n}
\DeclareSymbolFont{largesymbols}{OMX}{cmex}{m}{n}
\usepackage{geometry}
\usepackage{xcolor}
\usepackage{soul}
\usepackage{array}
\usepackage{bbding}
\usepackage{multirow}
\usepackage[fleqn]{amsmath}
\usepackage{amssymb}
\usepackage{amsthm}
\usepackage{graphicx}
\usepackage{natbib}
\usepackage{lscape}
\usepackage[hyphens]{url}
\usepackage{hyperref}
\usepackage{comment}
\usepackage{bm}
\usepackage[title]{appendix}
\usepackage{longtable}
\usepackage{mathtools}
\usepackage{chngcntr}
\usepackage{authblk}
\usepackage{babel}
\usepackage{dsfont}
\usepackage{subcaption}
\usepackage{fancyhdr}
\usepackage{cleveref}
\usepackage{abstract}
\usepackage{booktabs}
\usepackage{tabularx}   
\usepackage{array}      

\makeatletter
\allowdisplaybreaks
\date{}

\fancypagestyle{plain}{%
	\fancyhf{} 
	
	\chead{\footnotesize \textcolor{gray}{\textit{\shortauthors -- \runningtitle}}}
	\rhead{\footnotesize \textcolor{gray}{\thepage}}
}
\makeatletter
\def\blfootnote{\xdef\@thefnmark{}\@footnotetext}
\makeatother

\makeatletter
\def\titlepageext{
	\begin{center}	
	{\parindent0pt
		\rule{0.9\textwidth}{1pt}
		\begin{minipage}[t]{0.25\textwidth}
			\small {\it Keywords:}\\
			\keyword
		\end{minipage}%
		\hspace{3mm}
		\begin{minipage}[t]{0.6\textwidth}
			\small \abstract
		\end{minipage}%
		
		\rule{0.9\textwidth}{2pt}
	}
	\end{center}

	\blfootnote{* Corresponding author. E-mail address: \href{mailto:\corresemail}{\corresemail}.}
}
\makeatother

\usepackage[mathlines, pagewise]{lineno}
\usepackage{etoolbox} 

\newcommand*\linenomathpatchAMS[1]{%
	\expandafter\pretocmd\csname #1\endcsname {\linenomathAMS}{}{}%
	\expandafter\pretocmd\csname #1*\endcsname{\linenomathAMS}{}{}%
	\expandafter\apptocmd\csname end#1\endcsname {\endlinenomath}{}{}%
	\expandafter\apptocmd\csname end#1*\endcsname{\endlinenomath}{}{}%
}

\expandafter\ifx\linenomath\linenomathWithnumbers
\let\linenomathAMS\linenomathWithnumbers
\patchcmd\linenomathAMS{\advance\postdisplaypenalty\linenopenalty}{}{}{}
\else
\let\linenomathAMS\linenomathNonumbers
\fi

\linenomathpatchAMS{gather}
\linenomathpatchAMS{multline}
\linenomathpatchAMS{align}
\linenomathpatchAMS{alignat}
\linenomathpatchAMS{flalign}

\nolinenumbers

\newtheorem{theorem}{Theorem}

\newtheorem{definition}{Definition}
\newtheorem{lemma}{Lemma}

\title{\textbf{Transit-MP: Transit-Prioritized Max-Pressure Control in Sparse Connected Vehicle Environments}}

\def\shortauthors{Tan et al.}
\def\runningtitle{Transit-MP}

\author[a,b]{Chaopeng Tan}
\author[c,$\ast$]{Hao Liu}
\author[a]{Dingshan Sun}
\author[a]{Marco Rinaldi}
\author[a]{Hans van Lint}
\affil[a]{Department of Transport and Planning, Delft University of Technology, Gebouw 23, Stevinweg 1, Delft, 2628CN, The Netherlands}
\affil[b]{Chair of Traffic Process Automation, Technische Universität Dresden, Hettnerstraße 3, Dresden, 01069, Germany}
\affil[c]{Department of Civil and Environmental Engineering, University of Maine, 168 College Ave., Orono, ME 04469, USA}

\def\corresemail{hao.liu1@maine.edu}

\def\abstract{
Max-pressure (MP) control stands out among real-time network traffic signal control methods due to its simplicity, decentralized nature, and theoretical stability. However, existing MP control methods have limited consideration of public transportation and do not address the network stability problem of transit-prioritized MP in partially connected vehicle (CV) environments. In this study, we propose Transit-MP, which realizes transit-prioritized MP control in partially CV environments by considering real-time vehicle occupancy and the impact of transit dwell at stations. Theoretically, we demonstrate that Transit-MP, while using different traffic state measures for upstream and downstream links for pressure calculation, still achieves road network stability even in partially CV environments. 
Note that for MP controllers in sparse CV environments, some movements may have missing CV observations, leading to link spillovers, which create the queue starvation phenomenon: a movement no longer receives the green phase despite the queue spillover. Therefore, we further propose a modified Transit-MP (mTransit-MP) that incorporates historical traffic data to address this issue. We rigorously prove that the proposed mTransit-MP can effectively avoid the queue starvation phenomenon.
Experimental results on a real-world corridor in Amsterdam with 15 transit lines and 31 stations show that our method significantly reduces the real-time vehicle and spillover count, and improves delays for both private vehicles and transit vehicles compared to a state-of-the-art MP controller for transit signal priority. In sparse CV environments, our mTransit-MP is effective in mitigating link spillovers while enhancing the overall performance of multi-modal traffic. 
}

\def\keyword{Maxpressure control \\ Transit signal priority \\ Connected vehicle \\
Network stability \\ Transit dwell}

\begin{document}
\maketitle
\titlepageext

\section{Introduction}
Traffic signal control is a critical component in managing urban traffic flows, directly impacting travel efficiency, fuel consumption, and urban air pollution \citep{guo2019urban}. Among real-time network traffic signal control methods, the Max-Pressure (MP) or Back-Pressure (BP) control algorithm stands out due to its simplicity, decentralized nature, and theoretical stability characteristic \citep{varaiya2013max}. The basic idea of an MP control is to switch the green phase to those movements exerting the highest pressure on the road network, where the movement is defined as a pair of incoming and outgoing links, and the movement pressure is calculated as the weighted traffic measure difference between the paired links. Unlike centralized network traffic signal control methods requiring extensive communication infrastructure and high computational capability \citep{lin2012efficient, yan2019network, wang2020optimizing}, MP control operates effectively with local traffic state information on the connecting links and allows each intersection to make independent signal decisions, making it scalable and cost-effective. Furthermore, the guarantees of queue stability and throughput optimality for MP control ensure that it can maintain efficient traffic flow within admissible demand regions \citep{varaiya2013max}. 

Existing studies on MP control can be generally categorized into two main streams: variants of MP control and extensions of MP control. The former stream of MP control variants focuses on investigating the use of various traffic measures for calculating movement pressure while ensuring network queue stability. The early Q-MP proposed by \citet{varaiya2013max} uses queue length, which is essentially the number of vehicles on the link due to the assumption of the point queue model, for pressure calculation. In the case of considering the priority of links with different features, such as dedicated bus lanes or short links, the weighted queue length can also be used. However, Q-MP assumes infinite queue capacities for stability analysis, which deviates from the real world. Then, \citet{gregoire2014capacity} proposed Capacity-Aware MP (CA-MP) to incorporate link capacity constraints, but its impact on road network stability has not been demonstrated. Nonetheless, these early models relied on aggregated traffic measures, such as vehicle counts, due to the limited capabilities of early detection technology.

Recent advancements in information technology have enabled real-time vehicle-to-everything (V2X) communications, allowing connected vehicles (CVs) to provide real-time location and speed data. Unlike fixed-location detectors, CVs offer two-dimensional spatiotemporal traffic observations, enhancing traffic state estimation \citep{zheng2017estimating, cao2021day, tan2019cycle, tan2021cumulative} and traffic flow management \citep{yao2019optimization,tan2025connected_robust, tan2025connected_stochastic,tan2024privacy, wang2024traffic}. 
Consequently, recent studies have leveraged CV data to calculate
movement pressure for MP control. For example, \citet{li2019position} introduced Position-Weighted MP (PW-MP), which gives more weight to queues near road inlets to mitigate spillback risks. Furthermore, \citet{wang2022learning} proposed a learning-curve-based MP (Learning-MP) that employs reinforcement learning to assign distinct weights to stopped and moving vehicles for optimized performance.

Despite these advances, the above studies only account for the instantaneous spatial distribution of vehicles, which can cause excessive delays on branch roads in unbalanced flow scenarios \citep{wu2018delay}. To address this, some studies have explored time-cumulative vehicle measures. For example, \citet{wu2018delay} proposed a Head-of-Line Delay-based MP (HD-MP) to reduce delays for smaller queues, while \citet{mercader2020max} validated that MP with average travel time for pressure calculation, termed TT-MP, performs well in field tests. Further refinements include delay-based MP (D-MP) considering cumulative queue lengths during the last decision horizon and total-delay-based MP (TD-MP) using the total delay of vehicles \citep{liu2022novel, liu2023total}, though the network queue stability is unproven in the latter study.
It is worth noting that although most of these MP controls utilizing vehicle-level information are essentially based on CV information, most of them have only proved network stability in 100\% CV scenarios or have not investigated stability explicitly. Most recently, \cite{tan2025cvmp} proposed CV-MP for MP control in partially CV environments, where the link travel time of CVs is used for pressure calculation, thus accounting for both spatial-temporal information of traffic measures. In particular, the proof of network queue stability when vehicles are partially connected and heterogeneously distributed is pioneered. However, in low CV penetration rate scenarios, some links may spill over due to the long-time absence of CV observations, leading to destabilization of the road network and degrading the signal control performance. Therefore, \emph{MP control with missing traffic observations due to low CV penetration rates needs to be further addressed.}

The other stream of MP control focuses on extending MP control to wider applications. For instance, \citet{zaidi2016back} and \citet{chen2022backpressure} integrated dynamic vehicle routing with MP control for enhanced network performance, where the local movement pressure information is used to adjust route choices of vehicles in real-time. \citet{liu2024n} proposed a novel MP algorithm that incorporates regional traffic states into the MP framework, thus achieving perimeter metering control at the boundary intersections. By modifying the calculation of pressure, \citet{xu2024smoothing} and \citet{ahmed2024c} extended MP control for arterial coordination. 
Given that signalized intersections in practice usually need to accommodate multi-modal traffic flows, \citet{xu2024ped} and \citet{liu2024max} extended MP control with consideration of the need for pedestrians crossing at intersections. 
Regarding transit vehicles, MP control, considering transit signal priority (TSP) is still in its infancy. Only a few MP control studies have achieved TSP. For instance, \citep{xu2022integrating} added hard constraints to Q-MP to force the triggering of transit signals for transit prioritization, which was shown to still maintain the network queue stability. Nevertheless, this study assumes dedicated lanes for buses. \citep{ahmed2024occ} proposed to weight the queue-based pressure of movements with average occupancy of vehicles on incoming links, i.e., OCC-MP, which prioritizes high-occupancy vehicles in mixed traffic scenarios. However, OCC-MP assumes constant occupancy for stability analysis. Though these studies have shown promising results in prioritizing transit vehicles in certain conditions, they do not consider more realistic transit operation scenarios where transit vehicles dwell at stops for passengers to board and alight. This will result in dynamic changes in transit vehicle occupancy along links, and dwelling at stops will have additional impacts on traffic operations. In addition, although the performance of the model in partially CV scenarios is tested, the stability of MP control considering TSP in partially CV scenarios is unproven.
Therefore, \emph{we believe MP control for TSP in sparse CV environments, which explicitly considers actions of more realistic transit operation, is worth investigating.}

This study proposes Transit-MP, a controller scheme designed to achieve transit-prioritized MP control in partially connected vehicle environments. The \emph{major contributions} are threefold. 1) We propose to use the occupancy-weighted link travel time of vehicles as the traffic measure for MP control, which takes into account both the spatial distribution (e.g., vehicle position distribution) and time-accumulated measures (e.g., delay) of vehicles and achieves TSP by prioritizing high-occupancy vehicles. 2) We prove that Transit-MP control while prioritizing transit vehicles achieves network queue stability even in partially CV environments. The challenge in this stability proof lies in the first-time use of different traffic states for upstream and downstream pressure calculations. 3) We address the short-link spillover problem, which is caused by missing real-time CV observations in sparse CV environments, by modifying Transit-MP with historical traffic data.

\section{Preliminaries}
\subsection{Network definition}
We model a signalized network as shown in Fig.~\ref{fig: network}. The set of signalized intersections, i.e., nodes, is denoted by $\mathcal{N}$. For a node $n \in \mathcal{N}$, the set of controlled movement is denoted by $\mathcal{M}_{n}$, which is indexed by a pair of incoming and outgoing links $(i,o)$. {Correspondingly, we use $n'$ to denote a neighboring node of $n$ that is connected by the link $o$. Then, the related movements of node $n'$ are indexed by $(o,k)$.} The signal state of movement $(i,o)$ is denoted by a binary decision variable $s_{i,o} \in \bm{s}_n$, where $s_{i,o} = 1$ indicates the green state of the signal of movement $(i,o)$ and the red state otherwise. We use $\bm{s}_n$ to denote the column vector of the signal state of node $n$ and $\bm{s}$ to denote the combination of $\bm{s}_n$ on the whole network. Correspondingly, we have $\mathcal{S}_n$ and $\mathcal{S}$ to denote the feasible space of signal states. We use $c_{i,o}$ to denote the saturated flow rate of movement $(i,o)$ at the stopline position and $r_{i,o}$ to denote the corresponding turning ratio. 

For stability analysis, we introduce fictitious source links and fictitious nodes to load traffic demand from exogenous real links, as \citet{li2019position} and \cite{tan2025cvmp} have done. The set of fictitious nodes is denoted by $\mathcal{F}$. In particular, the fictitious source links have no physical length, so they are assumed to have infinite jam densities. The traffic flow is modeled as point queues concentrated at source nodes, where the queue lengths can accumulate without bound. Note that due to the minimum vehicle time headway constraint, the inflow and outflow rates remain finite. Since real links have finite jam densities, the network is said to be unstable if the congestion spills over to the fictitious sourced links and the queue length is accumulated infinitely on the fictitious sourced link. 

Considering a mixed CV and non-CV (NV) scenario, on the incoming link of movement $(i,o)$, $\mathcal{V}_{i,o}^{cv}$, $\mathcal{V}_{i,o}^{hv}$, and $\mathcal{V}_{i,o}$ are used to denote the sets of CVs, NVs and all vehicles, respectively. Obviously, $\mathcal{V}_{i,o} = \mathcal{V}_{i,o}^{cv} \bigcup \mathcal{V}_{i,o}^{hv}$. 
For transit vehicles, we assume that they are all connected and belong to $\mathcal{V}_{i,o}^{cv}$. For all CVs, it is assumed that the occupancy information $p_v$, i.e., the number of passengers (including the driver) riding on the vehicle $v$, can be shared for traffic signal control. Note that in practice, even though the real-time occupancy information of CVs is not available, we can use historical statistics as an alternative. For example, the occupancy of private cars can be assumed as a constant value, and that of public transit vehicles can be estimated from the historical boarding and alighting data between stations \citep{kuchar2023passenger}. Since most of the CVs are currently vehicles using map navigation services, whose routes are planned in advance. Therefore, in this study, it is also assumed that their turn intentions at intersections are known.

\begin{figure}[ht]
    \centering
    \includegraphics[width=0.9\textwidth]{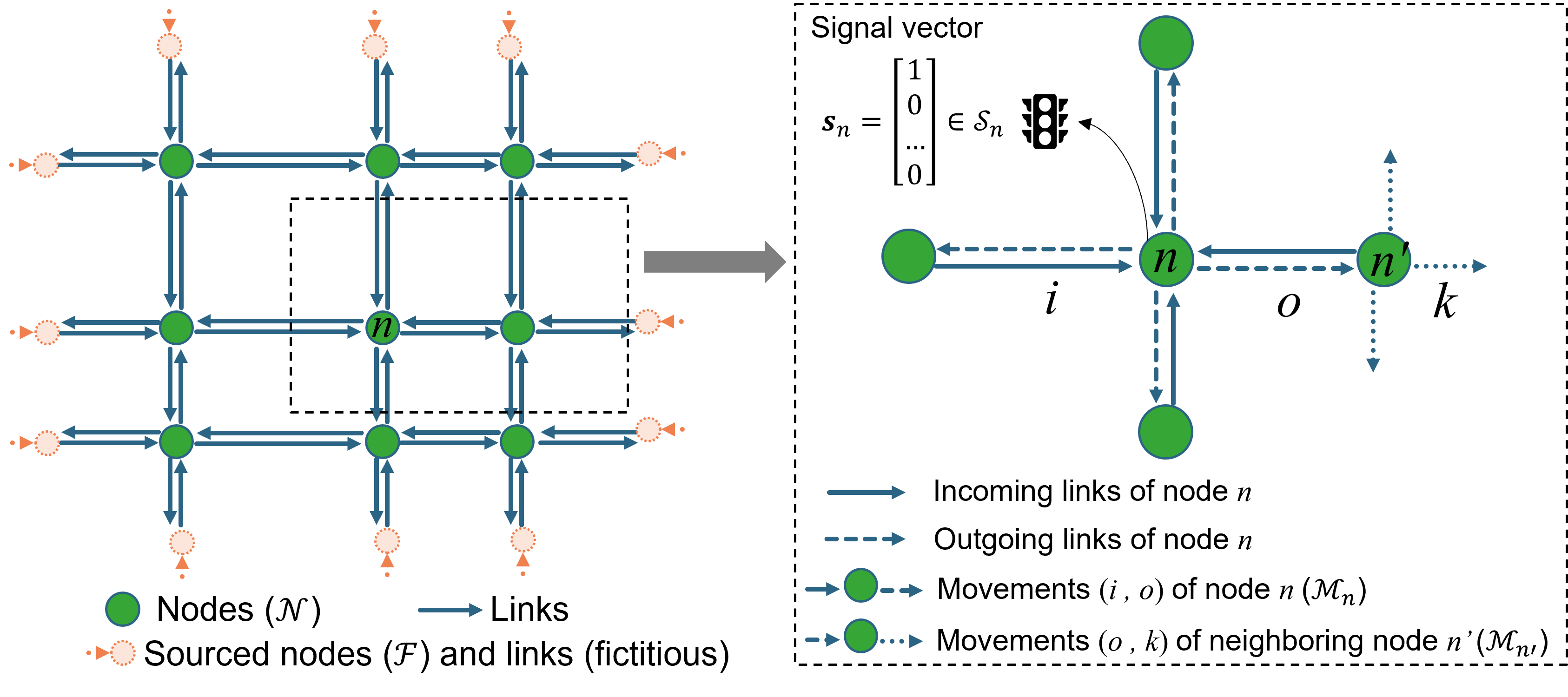}
    \caption{Network definition}
    \label{fig: network}  
\end{figure}

\subsection{CV-MP}
In the study by \citet{tan2025cvmp}, an MP control using CV data only, i.e., CV-MP, is proposed, in which the real-time link travel time of CVs is used for pressure calculation, thus accounting for both spatial-temporal information of traffic measures. Nevertheless, CV-MP is (i) designed for private traffic only, and (ii) it may suffer from spillovers on short links, i.e., fail to stabilize the network queue, in sparse CV environments due to the long-time absence of CV observations. 
In this study, we will extend CV-MP to multi-modal traffic scenarios and address the possible spillover problem in sparse CV environments. 

For a smooth read, we briefly introduce the control policy of CV-MP: 
\begin{align}
    \bm{s}^*(t) &= \arg \max_{\bm{s} \in \bm{S}} \sum_{n \in \mathcal{N}} \left( \sum_{\forall (i,o) \in \mathcal{M}_n} s_{i,o}(t) c_{i,o} \left( \sum_{v \in \mathcal{V}_{i,o}^{cv}} \tau_v(t) - \sum_{(o,\forall k) \in \mathcal{M}_{n'}} r_{o,k}(t) \sum_{v \in \mathcal{V}_{o,k}^{cv}} \tau_v(t) \right) \right), \label{eq: CV_MP}
\end{align}
where $\tau_v$ denotes the normalized link travel time of CV $v$ at the decision moment $t$ and
\begin{align}
    \tau_v = \frac{LTT_v}{\overline{ETT}_{i,o}}.
\end{align}
$LTT_v$ denotes the real-time link travel time of vehicle $v$ calculated from the moment the vehicle enters the link. $\overline{ETT}_{i,o}$ denotes the expected free flow travel time on link $i$ for movement $(i,o)$. 
The study by \citet{tan2025cvmp} has demonstrated that CV-MP can stabilize the road network queue if the following sufficient condition on CV observations is satisfied:
\begin{theorem}[Necessary condition of CV-MP stability in partially CV environments \cite{tan2025cvmp}] \label{thm: necessary condition}
    Let $\mathcal{M}_n^{p}$ denote the set of movements served by phase $p$ at intersection $n \in \mathcal{N}$. A necessary condition of CV-MP to stabilize the network queue in partially CV environments is that there must NOT exist any time $t^*$ such that 
    \begin{align}
        \mathcal{V}_{i,o}^{cv} = \varnothing \quad \text{and} \quad \rho_{i,o}(x,t^*) = \rho_{i,o}^{max} \quad \forall x \in [0,L_i] \quad \forall (i,o) \in \mathcal{M}_n^{p} \label{eq: necessary for CV-MP}
    \end{align}
    for all phases at all intersections, where $\rho_{i,o}(x,t^*)$ denotes the traffic density at position $x$ (measured from the inlet), moment $t^*$ on the incoming link of movement $(i,o)$ and $\rho_{i,o}^{max}$ denotes the maximum traffic density, and $L_i$ denotes the corresponding link length. 
\end{theorem}
Eq.~\eqref{eq: necessary for CV-MP} can be interpreted as at moment $t^*$, the traffic density of all movements in the phase $p$ has reached its maximum while no CVs are observed. Obviously, in such a case, these movements will no longer have CV observations; thus, the phase pressure will always be 0 despite the presence of spillover by NVs. Please see \citet{tan2025cvmp} for a more detailed proof of Theorem \ref{thm: necessary condition}. 

The opposite of Theorem \ref{thm: necessary condition} is to say that if the condition of Eq.~\eqref{eq: necessary for CV-MP} occurs, CV-MP can no longer stabilize the network queue. Note that, in sparse CV environments, there is a high probability that the condition of Eq.~\eqref{eq: necessary for CV-MP} will happen, especially for short links, as the traffic density of short links is easier to reach the maximum. 
In this study, when we extend CV-MP to multi-modal traffic scenarios, we will also design a modification mechanism to address this unsolved problem in \cite{tan2025cvmp}.

\section{Transit-MP in partially CV environments}
\subsection{Transit-MP}

Fig. \ref{fig: LTT} illustrates the approaching process of CVs on a link. For transit vehicles, they may dwell at transit stations for passengers to board and alight. Obviously, the signal priority for transit vehicles is not necessary until they leave the nearest station to the intersection. Here, we introduce a binary parameter $\beta_v$ to indicate whether the vehicle contributes to the pressure calculation according to their real-time position $x_v$ and vehicle type. 
\begin{align}
    \beta_v = \begin{cases}
        1, \text{ if } \alpha_v = 0, \\
        1, \text{ if } \alpha_v = 1, \text{ and } x_v \geq L_{ts}, \\
        0, \text{ if } \alpha_v = 1, \text{ and } x_v < L_{ts}.
    \end{cases} \label{eq: beta}
\end{align}
where $\alpha_v$ denotes the vehicle type; $\alpha_v = 0$ indicates that the vehicle is a private car and, otherwise, a transit vehicle; $L_{ts}$ denotes the position of the nearest transit station to the intersection; $x_v$ is the real-time position of CV $v$. Both positions are measured from the inlet of the link. $\beta_v = 1$ indicates that the vehicle will contribute to the pressure calculation. Eq.~\eqref{eq: beta} means that all connected private cars ($\alpha_v = 0$) will contribute to the pressure calculation, and connected transit vehicles will only contribute after they pass the nearest transit station, i.e., $x_v \geq L_{ts}$. 

Since the actual dwell time of transit vehicles is related to the number of boarding and alighting passengers \citep{kwesiga2025analysis}, we only determine $\beta_v$ based on the real-time location of transit vehicles, which can be achieved with simple connected vehicle functionality. However, if transit vehicles are required to run strictly according to the timetable, i.e., when departure times at stations can be accurately predicted, $\beta_v$ can be determined in a more precise manner (see Appendix \ref{appendix: beta} for details).

\begin{figure}[ht]
    \centering
    \includegraphics[width=0.7\textwidth]{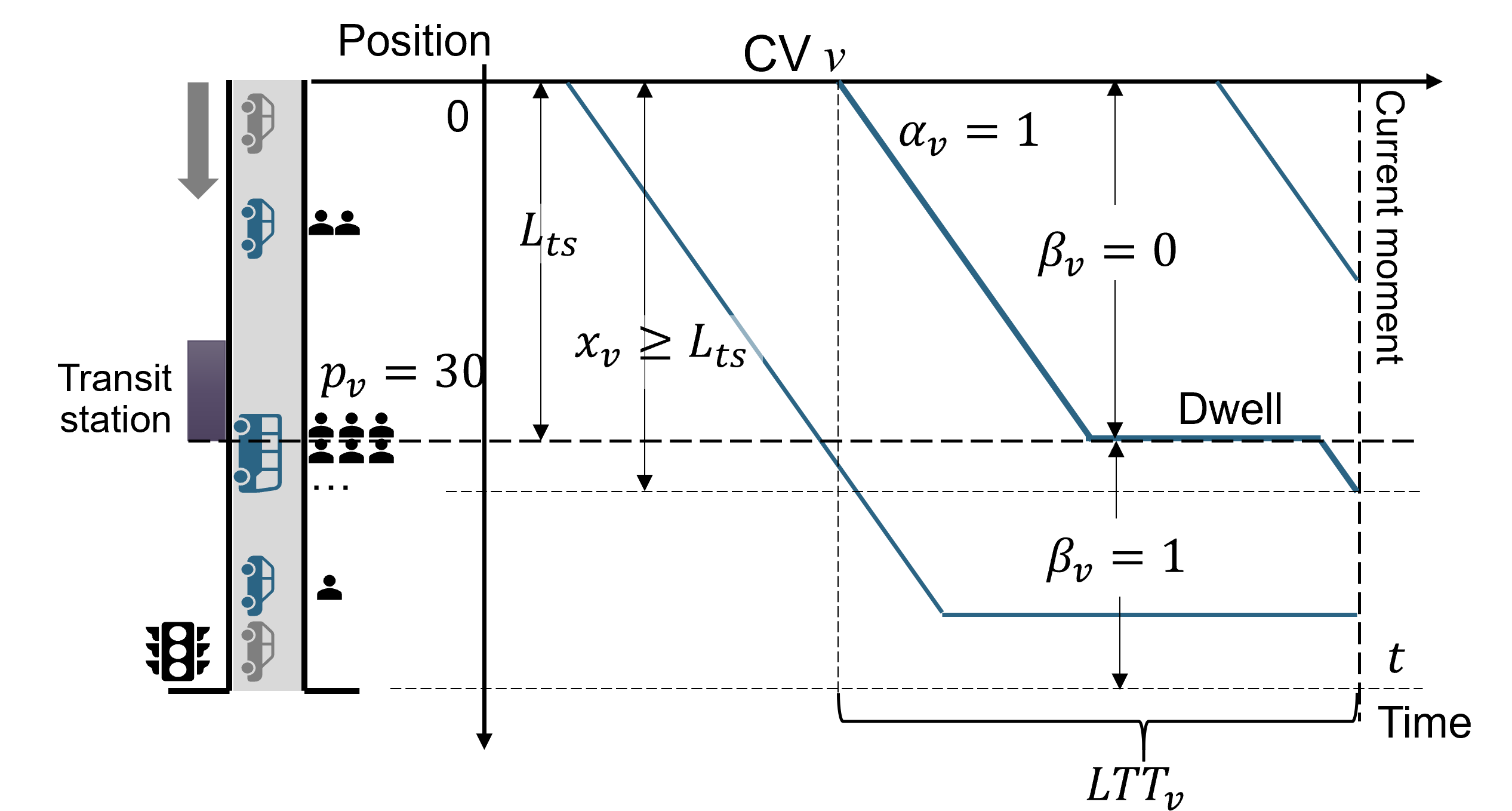}
    \caption{Vehicle information for $w_{i,o}(t)$}
    \label{fig: LTT}  
\end{figure}

Then, at time step $t$, the proposed Transit-MP in partially CV environments is applied by: 
\begin{align}
    \bm{s}^* &= \arg \max_{\bm{s} \in \mathcal{S}} \sum_{n \in \mathcal{N}} \left( \sum_{\forall( i, o) \in \mathcal{M}_{n}} s_{i,o} c_{i,o} \left(\sum_{v \in \mathcal{V}_{i,o}^{cv}} \beta_v p_v \tau_v - \sum_{ (o,\forall k)\in \mathcal{M}_{n'}} r_{o,k}\sum_{v \in \mathcal{V}_{o,k}^{cv}} \beta_v\tau_v \right)\right) \label{eq: transit MP} 
\end{align}
where we let 
\begin{align}
    c_{i,o} = \begin{cases}
        0 \quad \text{if} \quad \sum_{v \in \mathcal{V}_{i,o}^{cv}} \beta_v \tau_v - \sum_{ (o,\forall k)\in \mathcal{M}_{n'}} r_{o,k}\sum_{v \in \mathcal{V}_{o,k}^{cv}} \beta_v\tau_v <0 \\
        \hat{c}_{i,o} \quad \text{else}
    \end{cases}. \label{eq: c= 0 condition}
\end{align}
Recall that $p_v$ denotes the occupancy of the vehicle. The time factor $t$ is omitted for simplicity, $\hat{c}_{i,o}$ is the average saturated flow rate, and the other parameters are the same as in Fig.~\ref{fig: network}. In particular, Eq.~\eqref{eq: c= 0 condition} means that if the original pressure weight without occupancy information, i.e., $\sum_{v \in \mathcal{V}_{i,o}^{cv}} \beta_v \tau_v - \sum_{ (o,\forall k)\in \mathcal{M}_{n'}} r_{o,k}\sum_{v \in \mathcal{V}_{o,k}^{cv}} \beta_v\tau_v $ is negative, which indicates a longer queue on the outgoing link $o$ compared to the incoming link $i$ for the movement $(i,o)$, the saturated flow $c_{i,o}$ is forced to be 0. This correction is posed for our theoretical stability analysis in Section \ref{sssec: stability of transit-mp}.  
Given the optimal signal policy of Transit-MP, we have 
\begin{align} 
    s_{i,o}^* c_{i,o} \left(\sum_{v \in \mathcal{V}_{i,o}^{cv}} \beta_v \tau_v - \sum_{ (o,\forall k)\in \mathcal{M}_{n'}} r_{o,k}\sum_{v \in \mathcal{V}_{o,k}^{cv}} \beta_v\tau_v \right) \geq 0 \label{eq: positive pressure} \\
    s_{i,o}^* c_{i,o} \left(\sum_{v \in \mathcal{V}_{i,o}^{cv}} \beta_v p_v \tau_v - \sum_{ (o,\forall k)\in \mathcal{M}_{n'}} r_{o,k}\sum_{v \in \mathcal{V}_{o,k}^{cv}} \beta_v\tau_v \right) \geq 0 \label{eq: positive occupancy pressure}
\end{align}
for $\forall( i, o) \in \mathcal{M_N}$. The latter is obtained as $p_v \geq 1$. Eqs.~\eqref{eq: positive pressure} and \eqref{eq: positive occupancy pressure} ensure that the movement pressure is non-negative. 

Compared to CV-MP presented in Eq. \ref{eq: CV_MP}, the proposed Transit-MP further incorporates the real-time occupancy information of CVs and excludes the impact of transit dwell at stations. Essentially, Transit-MP prioritizes the movement with the greatest total passenger travel time from CV observations. 
In particular, as suggested by \cite{ahmed2024occ}, we only incorporated the occupancy information of the upstream movement for pressure calculation to avoid downstream supply issues. This is because the downstream term actually reflects the supply ability of the downstream link to accommodate vehicles, while a larger number of passengers downstream does not necessarily mean a larger number of vehicles on the link if there are high occupancy transit vehicles.

However, it is worth clarifying that, while OCC-MP by \citet{ahmed2024occ} also utilized occupancy information of upstream movements for TSP, they multiplied the average upstream occupancy $\overline{p}_{i,o}$ with the final movement pressure, i.e., $\overline{p}_{i,o}(\sum_{v \in \mathcal{V}_{i,o}} 1 - \sum_{ (o,\forall k)\in \mathcal{M}_{n'}} r_{o,k}\sum_{v \in \mathcal{V}_{o,k}} 1)$ in vehicular form. Besides, they only proved the stability of the OCC-MP at an isolated intersection in fully connected vehicle environments, given a very strong assumption that $\overline{p}_{i,o}$ is fixed over time, and the impact of transit dwell at stations is not considered. 
In our Transit-MP, the real-time vehicular occupancy information is only incorporated in the upstream traffic state calculation with the consideration of transit dwell at stations, i.e., $\sum_{v \in \mathcal{V}_{i,o}^{cv}} \beta_v p_v \tau_v$ for upstream $(i,o)$ and $\sum_{v \in \mathcal{V}_{o,k}^{cv}} \beta_v\tau_v$ for downstream $(o,k)$. In addition, we will rigorously prove that the proposed Transit-MP also achieves stability of the road network in partially CV environments with time-varying occupancy information. 

\subsection{Network stability of Transit-MP}

\subsubsection{Definitions and Lemmas}
Let $z_{i,o}$ denote the number of vehicles of the movement $(i,o)$. Then it can be represented in density form as below:
\begin{align}
    z_{i,o}(t) = \begin{cases}
            \rho_{(i,o)}(t), \quad \text{ if } (i,o) \in \mathcal{M_F} \\
            \int_0^{L_i} \rho_{(i,o)}(x,t) \mathrm{d}x, \quad \text{ if } (i,o) \in \mathcal{M_N}
        \end{cases} \label{eq: number of vehicles}
\end{align}
where $\rho_{(i,o)}(x,t)$ denotes the traffic density at position $x$ and moment $t$ of movement $(i,o)$. As fictitious links have no physical length, their traffic density $\rho_{(i,o)}(t)$ is independent of position $x$ and can be increased infinitely. 

In density form, the Transit-MP policy in partially CV environments is rewritten as
\begin{align}
    \textit{Traffic} & \textit{ density form:} \nonumber \\
    \bm{s}^*(t) & = \arg \max_{\bm{s} \in \bm{S}} \sum_{n \in \mathcal{N}} \left( \sum_{\forall (i,o) \in \mathcal{M}_n} s_{i,o}(t) c_{i,o} \left( \int_0^{L_i} \tau^u_{i,o}(x,t) \rho_{i,o}^{cv}(x,t) \mathrm{d}x \right. \right.\nonumber \\
    & \qquad \qquad \qquad \qquad \qquad \qquad \qquad \qquad- \left.\left.\sum_{(o,\forall k) \in \mathcal{M}_{n'}} r_{o,k}(t) \int_0^{L_o} \tau^d_{o,k}(x,t) \rho_{o,k}^{cv}(x,t) \mathrm{d}x \right) \right), \label{eq: Transit-MP, density}
\end{align}
where $\tau^u_{i,o}(x,t)$ and $\tau^d_{o,k}(x,t)$ are the weights on traffic flow density of upstream and downstream movements, respectively, and
\begin{align}
    &\int_0^{L_i} \tau^u_{i,o}(x,t) \rho_{i,o}^{cv}(x,t) \mathrm{d}x = \sum_{v \in \mathcal{V}_{i,o}^{cv}} \beta_v p_v \tau_v \triangleq w^{u,cv}_{i,o}(t), \label{eq: density upstream weight}\\
    &\int_0^{L_i} \tau^d_{o,k}(x,t) \rho_{o,k}^{cv}(x,t) \mathrm{d}x = \sum_{v \in \mathcal{V}_{o,k}^{cv}} \beta_v \tau_v \triangleq w^{d,cv}_{o,k}(t), \label{eq: density downstream weight}
\end{align}
where $\rho_{i,o}^{cv}(x,t)$ and $\rho_{o,k}^{cv}(x,t)$ denote the traffic flow density of CVs; $w^{u,cv}_{i,o}(t)$ and $w^{d,cv}_{o,k}(t)$ represents the traffic states of the movements for pressure calculation of Transit-MP. 

\begin{lemma} [Upper bounds of related parameters \citep{tan2025cvmp}] \label{lemma: upper bound}
    For any $x \in [0, L_i]$, $t >0$, and $(i,o) \in \mathcal{M_N}$, the following inequalities hold if the weight $y_{i,o}(x,t)$ satisfies $0 \leq y_{(i,o)}(x,t) \leq y_{(i,o)}^{max} < \infty$.
    \begin{align}
        &\rho_{i,o}(x,t) \leq \rho_{i,o}^{max} < \infty, \label{eq: upper bounds 1}\\ 
        &q_{i,o}(x,t) \leq q_{i,o}^{max} < \infty, \label{eq: upper bounds 2} \\
        &|\partial_x q_{i,o}(x,t)| \leq \dot{q}_{i,o}^{max} < \infty \\
        &\int_0^{L_i}\int_0^{L_i} \rho_{i,o}(x,t) \rho_{i,o}(x',t) \mathrm{d} x' \mathrm{d}x \leq (L_i \rho_{i,o}^{max})^2 < \infty, \label{eq: upper bounds 4} \\
        &\int_0^{L_i} y_{i,o}(x,t)\rho_{i,o}(x,t) \mathrm{d}x \leq y_{i,o}^{max} L_i \rho_{i,o}^{max} , \label{eq: upper bounds 5} \\
        &\int_{0_{+}}^{L_{i-}} -\frac{\partial q_{i,o}(x,t)} {\partial x}  \mathrm{d} x \leq q_{i,o}^{max} , \label{eq: upper bounds 6} \\
        &\int_{0_{+}}^{L_{i-}} - y_{i,o}(x',t)\frac{\partial q_{i,o}(x',t)} {\partial x'}  \mathrm{d} x' \leq y_{i,o}^{max} L_i \dot{q}_{i,o}^{max}. \label{eq: upper bounds 7}
    \end{align}
    where $q_{i,o}(x,t)$ denotes the flow rate at position $x$ and moment $t$ of movement $(i,o)$; $\dot{q}_{i,o}(x,t) = \frac{\partial q_{i,o}(x,t)} {\partial x}$; $\rho_{i,o}^{max}$, $q_{i,o}^{max}$, and $\dot{q}_{i,o}^{max}$ are upper bounds of corresponding parameters, which are positive constants. 
\end{lemma}
The detailed proof of Lemma \ref{lemma: upper bound} can be found in \cite{ tan2025cvmp}. As real links have physical lengths, their traffic state parameters, $\rho_{i,o}(x,t)$ and $q_{i,o}(x,t)$ are naturally bounded, which gives the first and second inequalities. The third inequality is due to the gradual changes in traffic flow. The remaining inequalities can be easily derived from the first three inequalities.

Recall that $\mathcal{S}$ denotes the feasible space of network signal states $\bm{s}$ under predefined signal phase constraints. We further use $\mathcal{S}^{ch}$ to denote the convex hull of $\mathcal{S}$. Then, the admissible demand region $\bm{\Lambda}$ of the traffic network is defined as
\begin{definition} [Admissible demand region] \label{def: admissible}
    For a signalized network given feasible signal space $\mathcal{S}$, turning ratio $\bm{r}$, and saturated flow rate $\bm{c}$, the admissible demand region $\bm{\Lambda}$ of the traffic network is defined as:
    \begin{align}
        \bm{\Lambda} =\{\bm{a} \mid \bm{a} \preceq (\bm{I} - \bm{r})\bm{c}\bar{\bm{s}}, \quad \exists \bar{\bm{s}} \in \bm{S}^{co}\}, \label{eq: admissible}
    \end{align}
    where $\bm{a} \in \mathbb{R}^{|\mathcal{M_{F \cup N}}| \times 1}$ denotes the column vector of the average exogenous arrival demands of the network; $\bar{\bm{s}}$ denotes the long term average of signal state $\bm{s}$, i.e., $\bar{\bm{s}} = \lim_{T \rightarrow \infty}\frac{1}{T}\sum_{t=1}^T \bm{s}$; 
    $\bm{r}$ denotes the matrix of turning ratio $r_{i,o}$ and $\bm{c}$ is a diagonal matrix of saturated flow rate $c_{i,o}$. Both $\bm{r}, \bm{c} \in \mathbb{R}^{|\mathcal{M_{F \cup N}}| \times |\mathcal{M_{F \cup N}}|}$. 
\end{definition}
The definition \ref{def: admissible} implies that we can always find a $\epsilon > 0$ that makes
\begin{align}
    \bm{a} - (\bm{I} - \bm{r})\bm{c}\bar{\bm{s}} \preceq -\epsilon \bm{1}, \quad \exists \epsilon > 0 \label{eq: admissible 1}
\end{align}
for $\bm{a} \in \bm{\Lambda}^{int}$, where $\bm{\Lambda}^{int}$ denotes the interior of $\bm{\Lambda}$.

As demonstrated in \citet{tan2025cvmp}, if the exogenous demand $\bm{a} \in \bm{\Lambda}^{int}$, the vehicle link travel time $\tau_v$ is upper bounded by a positive constant $\tau^{max}$. Then, we have the following lemma about the weight of Transit-MP. 
\begin{lemma} [Properties of weights of Transit-MP] \label{lemma: weights}
    If the exogenous demand $\bm{a} \in \bm{\Lambda}^{int}$, we have
    \begin{align}
        & 0 \leq \tau^u_{i,o}(x,t) \leq \tau^{u,max}_{i,o} < \infty, \\
        & \tau^u_{i,o}(0,t) = 0, \\
        & 0 \leq |\partial_t \tau^u_{i,o}(x,t)| \leq \dot{\tau}^{u,max}_{i,o} < \infty.
    \end{align}
    where $\partial_t$ denotes the partial differentiation in terms of $t$. Similar properties also apply to $\tau^d_{i,o}(x,t)$ with upper bounds $\tau^{d,max}_{i,o}$ and $\dot{\tau}^{d,max}_{i,o}$. 
\end{lemma}
\begin{proof}
    As $\tau^u_{i,o}(x,t)\rho_{i.o}(x,t)$ is essentially the density form of $\beta_v p_v \tau_v$, the first equality is equivalent to proving that $\beta_v p_v \tau_v$ is upper bounded. Recall that $\beta_v$ is a binary variable and $p_v$ indicates the occupancy of the vehicle, which is obviously upper bounded by the vehicle capacity $p^{max}$. Then we have $\beta_v p_v \tau_v \leq p^{max} \tau^{max} < \infty$, which equally means there exists a positive constant $\tau^{u,max}_{i,o}$ that makes $0 \leq \tau^u_{i,o}(x,t) \leq \tau^{u,max}_{i,o} < \infty$.
    
    Recall that $\tau_v = {LTT_v}/{\overline{ETT}_{i,o}}$ and $LTT_v > 0$ only when $x_v > 0$. That is to say, the vehicle participates in pressure calculation only when $x_v>0$, which obviously leads to $\tau^u_{i,o}(0,t) = 0$. 

    Regarding $|\partial_t \tau^u_{i,o}(x,t)|$, it is equivalent to considering $|\partial_t \tau_v|$, i.e., changes in vehicle link travel time. Considering a small period $\Delta t$, if the vehicle just enters in or stays on the link, $|\tau_v(t+\Delta t) - \tau_v(t)| \leq \Delta t$; if the vehicle leaves the link, $|\tau_v(t+\Delta t) - \tau_v(t)| = \tau_v(t) \leq \tau^{max}$. This leads to the existence of a positive constant $\dot{\tau}^{u,max}_{i,o}$ that makes $|\partial_t \tau^u_{i,o}(x,t)| \leq \dot{\tau}^{u,max}_{i,o} < \infty$. 
    
    Similar analysis also applies to $\tau^d_{i,o}(x,t)$. Thus, Lemma \ref{lemma: weights} is proved. 
    
\end{proof}

\begin{definition} [Traffic network stability \citep{neely2022stochastic, li2019position, tan2025cvmp}] \label{def: stability}
    Given a traffic signal control policy, the traffic network is said to be strongly stable if the following condition on the total queues holds:
    \begin{align}
        &\limsup_{T \rightarrow \infty} \frac{1}{T} \int_0^T \mathbb{E}\left( \sum_{(i,o) \in \mathcal{M_{F\cup N}}} z_{(i,o)}(t) \right) \mathrm{d}t < \infty  \label{eq: stability-definition}
    \end{align}
\end{definition}
Definition \ref{def: stability} indicates that the traffic network is stable if the signal control policy ensures that the network queues will not grow infinitely in the long term. As those real links have a physical upper-bound jam capacity, the network queues for movement $(i,o) \in \mathcal{M_N}$ are naturally bounded. Then, the network is unstable when the congestion spills over to fictitious links and the queue of movement $(i,o) \in \mathcal{M_F}$ grows infinitely. 

For brevity, hereafter we use bold to indicate the corresponding vector or matrix of a variable, e.g., $\bm{\rho}$ denotes the vector of $\rho_{i,o}$ of all movement $(i,o)$ over the network.

\begin{definition} [Lyapunov function] \label{def: lyapunov}
    Given the traffic flow density $\bm{\rho}(t)$ of the network, a Lyapunov function is defined as
    \begin{align}
        V(\bm{\rho}(t)) \equiv \frac{1}{2} \sum_{(i,o) \in \mathcal{M}_{\mathcal{F}}}  \rho_{i,o}(t)^2 + \frac{1}{2} \sum_{(i,o) \in \mathcal{M}_{\mathcal{N}}} \int_0^{L_i}\int_0^{L_i} \left(\tau_{i,o}^d(x,t)+\tau_{i,o}^d(x',t)\right) \rho_{i,o}(x,t) \rho_{i,o}(x',t) \mathrm{d} x' \mathrm{d}x,  \label{eq: lyapunov}
    \end{align} 
    where $\tau_{i,o}^d(x,t)$ is the weight that varies spatiotemporally on traffic flow density (corresponding to the downstream movement) and $\tau_{i,o}^d(x,t) \geq 0, \ \forall \ (i,o) \in \mathcal{M_N}, \ x\in [0, L_i], \, t \geq 0$, which ensures that $V(\bm{\rho}(t)) \geq 0$ and $V(\bm{\rho}(t)) = 0$ if and only if $\bm{\rho}(t)=0$.
\end{definition}

The Lyapunov function is essentially the total sum of the double integral of traffic flow density with certain weights. In our Transit-MP, traffic states of upstream and downstream movements are calculated differently, i.e., upstream incorporates vehicle occupancy and downstream does not. The Lyapunov function defined in this study uses the downstream weight only, which is essentially the sum of the double integral of the travel-time-weighted traffic flow density. 

Reproduced from \citet{li2019position} and \citet{tan2025cvmp}, a sufficient condition for network stability is derived as
\begin{lemma} [Sufficient condition for traffic network stability] \label{lemma: sufficient condition}
    Given $V(\bm{\rho}(t))$ defined in Eq.~\eqref{eq: lyapunov}, if there exist positive and upper bounded constants $K$ and $\epsilon'$, i.e., $0<K<\infty$ and $0<\epsilon'<\infty$, such that the Lyapunov drift 
    \begin{align}
    \mathbb{E}^{\bm{\rho}(t)} \left(\frac{\mathrm{d} V(\bm{\rho}(t))}{\mathrm{d} t} \right) \leq K - \epsilon'  \mathbb{E} \left(\sum_{(i,o) \in \mathcal{M_{F\cup N}}} z_{(i,o)}(t) \right) \label{eq: stability-condition}
    \end{align}
    holds for $\forall t\geq0$ and all possible $\bm{\rho}(t)$, the traffic network is stable by satisfying Eq.~\eqref{eq: stability-definition}.
\end{lemma}
The detailed proof of Lemma \ref{lemma: sufficient condition} can be found in \cite{varaiya2013max, li2019position, tan2025cvmp}, which can be easily obtained by integrating both sides of Eq.~\eqref{eq: stability-condition} over period $[0, T]$ with expectation and re-ordering the terms.

\subsubsection{Stability of Transit-MP} \label{sssec: stability of transit-mp}

In this section, we first demonstrate that the proposed Transit-MP, which prioritizes high-occupancy transit vehicles with consideration of the impact of transit stations, can stabilize the network queues in partially CV environments. Note that, though both \cite{li2019position} and \cite{tan2025cvmp} also considered continuous traffic flow dynamics for stability proof, this study differentiates from them as Transit-MP uses different spatial-temporal varying weights, i.e., $\tau^u_{i,o}$ and $\tau^d_{i,o}$, on upstream and downstream traffic state, respectively. 
\begin{theorem}[Stability of Transit-MP in partially CV environments] \label{thm: stability of transit-MP}
    Given the admissible demand region $\bm{\Lambda}$ in Definition \ref{def: admissible}, if the exogenous demand $\bm{a} \in \bm{\Lambda}^{int}$, the proposed Transit-MP presented in Eq.~\eqref{eq: transit MP} (or Eq.~\eqref{eq: Transit-MP, density} in density form) can strongly stabilize the traffic network queues in partially CV environments. 
\end{theorem}

\begin{proof}
    For brevity, the partial differentiation in terms of $t$ (or $x$) is denoted by $\partial_t$ (or $\partial_x$). If not specified, all variables indicate movement $(i,o)$, thus, the subscript is omitted. 

    \textbf{First, we simplify the Lyapunov drift.}
    According to Lemma \ref{lemma: sufficient condition}, it is equivalent to proving that Eq.~\eqref{eq: stability-condition} holds for all $t \geq 0$ under the control of Transit-MP. Following the idea of \citet{li2019position} and \citet{tan2025cvmp}, we first decompose and simplify $\mathbb{E}^{\bm{\rho}(t)} \left(\frac{\mathrm{d} V(\bm{\rho}(t))}{\mathrm{d} t} \right)$ based on the Leibniz integral rule and the product rule of partial differentiation as below, 
    \begin{align}
        \mathbb{E}^{\bm{\rho}(t)} \left(\frac{\mathrm{d} V(\bm{\rho}(t))}{\mathrm{d} t}\right) = & \underbrace{\sum_{\mathcal{M}_{\mathcal{F}}}\mathbb{E}^{\bm{\rho}(t)} \left( \rho(t)\frac{\mathrm{d} \rho(t)}{\mathrm{d} t} \right)}_{\delta_1} \nonumber \\
        & + \underbrace{ \sum_{\mathcal{M}_{\mathcal{N}}}\mathbb{E}^{\bm{\rho}(t)} \left( \int_0^{L_i}\int_0^{L_i} [\partial_t \tau^d(x,t)] \rho(x,t) \rho(x',t) \mathrm{d} x' \mathrm{d}x \right)}_{\delta_2} \nonumber \\
        & + \underbrace{ \sum_{\mathcal{M}_{\mathcal{N}}}\mathbb{E}^{\bm{\rho}(t)} \left(  \int_0^{L_i}\int_0^{L_i} \left(\tau^d(x,t)+\tau^d(x',t)\right) \rho(x,t) [\partial_t\rho(x',t)] \mathrm{d} x' \mathrm{d}x \right)}_{\delta_3} \label{eq: theorem 1 proof_1}
    \end{align}
    where $x'$ and $x$ are equivalent in status and can therefore be interchanged to cancel out 1/2. Considering continuous traffic flow dynamics \citep{li2019position}, we have
    \begin{align} 
        &\frac{\mathrm{d}\rho_{i,o}(t)}{\mathrm{d}t} = 
        a_{i,o}(t)-\min\{c_{i,o} s_{i,o}(t), \mu_{i,o}(t)\}, \quad \text{for} \quad (i,o) \in \mathcal{M_F}, \label{eq:dynamics-1} \\
        &\partial_t \rho_{i,o}(x,t) = - \partial_x q_{i,o}(x,t) \quad \text{for} \quad x\in(0,L_i), \quad (i,o) \in \mathcal{M_N}, \label{eq:dynamics-2}
    \end{align} 
where $a_{i,o}(t)$ denotes the exogenous arrival demand rate; $q_{i,o}(x,t)$ denotes the flow rate; $\mu_{i,o}(t)$ denotes the the effective serviceable demand rate that depends on both the egress density of the upstream link $i$ and the ingress density of the downstream link $o$, thus respecting both upstream vehicle availability and downstream storage constraints. 
At the boundary of real links, i.e., when $x=0$ and $x=L_i$, we have
\begin{align}
    &\partial_t \rho_{i,o}(0,t) = - \partial_x q_{i,o}(0,t) = \sum_{(\forall h,i)\in \mathcal{M}_{n''}} r_{i,o} \min\{ c_{h,i} s_{h,i}(t), \mu_{h,i}(t) \} - q_{i,o}(0,t),\label{eq:dynamics-3}\\ 
    &\partial_t \rho_{i,o}(L_i,t) = - \partial_x q_{i,o}(L_i,t) = q_{i,o}(L_i,t) - \min\{c_{i,o} s_{i,o}(t), \mu_{i,o}(t)\} \label{eq:dynamics-4}
\end{align}
for $(i,o) \in \mathcal{M_N}$, where $n''$ denotes the upstream node. 

Then, substituting Eq.~\eqref{eq:dynamics-1} into $\delta_1$, we have 
\begin{align}
    \delta_1 = &\sum_{\mathcal{M}_{\mathcal{F}}}\mathbb{E}^{\bm{\rho}(t)} \left( (a(t)-\min\{c s(t), \mu(t)\}) \rho(t) \right)
\end{align}
Based on Lemma \ref{lemma: upper bound} and Lemma \ref{lemma: weights}, we have 
$$[\partial_t \tau^d(x,t)] \rho(x,t) \rho(x',t) \leq |\partial_t \tau^d(x,t)| \rho(x,t) \rho(x',t) \leq \dot{\tau}^{d,max}\rho(x,t) \rho(x',t).$$ Thus, we have 
\begin{align}
    \delta_2 
    \leq & \sum_{\mathcal{M}_{\mathcal{N}}}\mathbb{E}^{\bm{\rho}(t)} \left( \dot{\tau}^{d,max} \int_0^{L_i}\int_0^{L_i} \rho(x,t) \rho(x',t) \mathrm{d} x' \mathrm{d}x \right) \nonumber \\
    \leq & \sum_{\mathcal{M}_{\mathcal{N}}}\dot{\tau}^{d,max}(L_i \rho^{max})^2 \triangleq K_2
\end{align}

Replacing $\partial_t\rho(x',t)$ based on Eq.~\eqref{eq:dynamics-3} and~\eqref{eq:dynamics-4}, we have 
\begin{align}
    \delta_3 = & \underbrace{\sum_{\mathcal{M}_{\mathcal{N}}} \mathbb{E}^{\bm{\rho}(t)} \left( \int_0^{L_i}\int_{0_+}^{L_{i-}} \left(\tau^d(x,t)+\tau^d(x',t)\right) \rho(x,t) [-\partial_x q(x',t)] \mathrm{d} x' \mathrm{d}x \right)}_{\delta_{3,1}} \nonumber \\
    & +\underbrace{\sum_{\mathcal{M}_{\mathcal{N}}} \mathbb{E}^{\bm{\rho}(t)} \left( [-\partial_x q(0,t)] \int_0^{L_i} \left(\tau^d(x,t)+\tau^d(0,t)\right) \rho(x,t) \mathrm{d}x \right)}_{\delta_{3,2}} \nonumber \\
    & +\underbrace{\sum_{\mathcal{M}_{\mathcal{N}}} \mathbb{E}^{\bm{\rho}(t)} \left( [-\partial_x q(L_i,t)] \int_0^{L_i} \left(\tau^d(x,t)+\tau^d(L_i,t)\right) \rho(x,t) \mathrm{d}x \right)}_{\delta_{3,3}}
\end{align}
where 
\begin{align}
    \delta_{3,1} 
     = & \sum_{\mathcal{M}_{\mathcal{N}}} \mathbb{E}^{\bm{\rho}(t)} \left( \int_0^{L_i} \tau^d(x,t) \rho(x,t) \left[\int_{0_+}^{L_{i-}}-\partial_x q(x',t)\mathrm{d} x'\right] \mathrm{d}x \right) \nonumber \\
     & + \sum_{\mathcal{M}_{\mathcal{N}}} \mathbb{E}^{\bm{\rho}(t)} \left( \int_0^{L_i} \rho(x,t) \left[\int_{0_+}^{L_{i-}}-\tau^d(x',t) \partial_x q(x',t)\mathrm{d} x'\right]  \mathrm{d}x \right) \nonumber \\
     \leq & \sum_{\mathcal{M}_{\mathcal{N}}} \mathbb{E}^{\bm{\rho}(t)} \left(q^{max} \int_0^{L_i} \tau^d(x,t) \rho(x,t) \mathrm{d}x \right) + \frac{1}{2}\sum_{\mathcal{M}_{\mathcal{N}}} \mathbb{E}^{\bm{\rho}(t)} \left( \tau^{d,max} L_i \dot{q}^{max} \int_0^{L_i} \rho(x,t) \mathrm{d}x \right) \nonumber \\
     \leq & \sum_{\mathcal{M}_{\mathcal{N}}} (q^{max} \tau^{d,max} L_i \rho^{max} + \tau^{d,max} L_i \dot{q}^{max} L_i \rho^{max}) \triangleq K_{3,1}
\end{align}
based on Lemma \ref{lemma: upper bound} and Lemma \ref{lemma: weights};
\begin{align}
    \delta_{3,2} =  & \sum_{\mathcal{M}_{\mathcal{N}}}\mathbb{E}^{\bm{\rho}(t)} \left(  [\sum_{(\forall h,i)\in \mathcal{M}_{n''}} r \min\{ c_{h,i} s_{h,i}(t), \mu_{h,i}(t) \} ] \int_0^{L_i} \left(\tau^d(x,t)+ \underbrace{\tau^d(0,t)}_{=0} \right) \rho(x,t) \mathrm{d}x \right) \nonumber \\
    & \underbrace{ - \sum_{\mathcal{M}_{\mathcal{N}}} \mathbb{E}^{\bm{\rho}(t)} \left( q(0,t) \int_0^{L_i} \left(\tau^d(x,t)+\tau^d(0,t)\right) \rho(x,t) \mathrm{d}x \right)}_{\leq 0} \nonumber \\
    \leq & \sum_{\mathcal{M}_{\mathcal{N}}}\mathbb{E}^{\bm{\rho}(t)} \left(  [\sum_{(\forall h,i)\in \mathcal{M}_{n''}} r \min\{ c_{h,i} s_{h,i}(t), \mu_{h,i}(t) \} ] \int_0^{L_i} \tau^d(x,t) \rho(x,t) \mathrm{d}x \right)
\end{align}
based on Lemma \ref{lemma: weights}; and 
\begin{align}
    \delta_{3,3} = &\sum_{\mathcal{M}_{\mathcal{N}}} \mathbb{E}^{\bm{\rho}(t)} \left( q_{i,o}(L_i,t) \int_0^{L_i} \left(\tau^d(x,t)+\tau^d(L_i,t)\right) \rho(x,t) \mathrm{d}x \right) \nonumber \\
    & - \sum_{\mathcal{M}_{\mathcal{N}}} \mathbb{E}^{\bm{\rho}(t)} \left( \min\{c s(t), \mu(t)\} \int_0^{L_i} \tau^d(x,t) \rho(x,t) \mathrm{d}x \right) \nonumber \\
    & \underbrace{- \sum_{\mathcal{M}_{\mathcal{N}}} \mathbb{E}^{\bm{\rho}(t)} \left( \min\{c s(t), \mu(t)\} \int_0^{L_i} \tau^d(L_i,t) \rho(x,t) \mathrm{d}x \right)}_{\leq 0} \nonumber \\
    \leq & \underbrace{\sum_{\mathcal{M}_{\mathcal{N}}} \mathbb{E}^{\bm{\rho}(t)} (2 q^{max}\tau^{d,max}L_i \rho^{max})}_{K_{3,3}} - \sum_{\mathcal{M}_{\mathcal{N}}} \mathbb{E}^{\bm{\rho}(t)} \left( \min\{c s(t), \mu(t)\} \int_0^{L_i} \tau^d(x,t) \rho(x,t) \mathrm{d}x \right).
\end{align}
As such, we have 
\begin{align}
    \delta_3 \leq & K_3 + \sum_{\mathcal{M}_{\mathcal{N}}}\mathbb{E}^{\bm{\rho}(t)} \left(  [\sum_{(\forall h,i)\in \mathcal{M}_{n''}} r \min\{ c_{h,i} s_{h,i}(t), \mu_{h,i}(t) \} ] \int_0^{L_i} \tau^d(x,t) \rho(x,t) \mathrm{d}x \right) \nonumber \\
    & - \sum_{\mathcal{M}_{\mathcal{N}}} \mathbb{E}^{\bm{\rho}(t)} \left( \min\{c s(t), \mu(t)\} \int_0^{L_i} \tau^d(x,t) \rho(x,t) \mathrm{d}x \right),
\end{align}
where $K_3 = K_{3,1} + K_{3,3}$. 

Then, Eq.~\eqref{eq: theorem 1 proof_1} is bounded by 
\begin{align}
    \mathbb{E}^{\bm{\rho}(t)} \left(\frac{\mathrm{d} V(\bm{\rho}(t))}{\mathrm{d} t}\right) \leq & K_2 + K_3 + \sum_{\mathcal{M}_{\mathcal{F}}}\mathbb{E}^{\bm{\rho}(t)} \left( (a(t)-\min\{c s(t), \mu(t)\}) \rho(t) \right) \nonumber \\
    & + \sum_{\mathcal{M}_{\mathcal{N}}}\mathbb{E}^{\bm{\rho}(t)} \left(  [\sum_{(\forall h,i)\in \mathcal{M}_{n''}} r \min\{ c_{h,i} s_{h,i}(t), \mu_{h,i}(t) \} ] \int_0^{L_i} \tau^d(x,t) \rho(x,t) \mathrm{d}x \right) \nonumber \\
    & - \sum_{\mathcal{M}_{\mathcal{N}}} \mathbb{E}^{\bm{\rho}(t)} \left( \min\{c s(t), \mu(t)\} \int_0^{L_i} \tau^d(x,t) \rho(x,t) \mathrm{d}x \right) \nonumber \\
    & \triangleq K_2 + K_3 + \eta
\end{align}

For brevity, we let $w^d(t) = \int_0^{L_i} \tau^d(x,t)\rho(x,t) \mathrm{d}x$ for downstream movements in $\mathcal{M_N}$. Similarly we also have $w^u(t) = \int_0^{L_i} \tau^u(x,t)\rho(x,t) \mathrm{d}x$ for upstream movements in later derivation. For movements in $\mathcal{M_F}$, we have $w^u(t) = w^d(t) = \rho(t)$. In vehicular form, we have
\begin{align}
    &w^d(t) = \int_0^{L_i} \tau^d(x,t)\rho(x,t) \mathrm{d}x \equiv \sum_{v \in \mathcal{V}} \beta_v \tau_v \nonumber \\
    &w^u(t) = \int_0^{L_i} \tau^u(x,t)\rho(x,t) \mathrm{d}x  \equiv \sum_{v \in \mathcal{V}} \beta_v p_v \tau_v,
\end{align}
which are essentially the traffic states for the pressure calculation of Transit-MP in fully CV environments. 

Then, in matrix form, we have
\begin{align}
    \eta =& \sum_{\mathcal{M}_{\mathcal{F}}}\mathbb{E}^{\bm{\rho}(t)} \left( (a(t)-\min\{c s(t), \mu(t)\}) \rho(t) \right) \nonumber \\
    & + \sum_{\mathcal{M}_{\mathcal{N}}}\mathbb{E}^{\bm{\rho}(t)} \left(  [\sum_{(\forall h,i)\in \mathcal{M}_{n''}} r \min\{ c_{h,i} s_{h,i}(t), \mu_{h,i}(t) \} ] \int_0^{L_i} \tau^d(x,t) \rho(x,t) \mathrm{d}x \right) \nonumber \\
    & - \sum_{\mathcal{M}_{\mathcal{N}}} \mathbb{E}^{\bm{\rho}(t)} \left( \min\{c s(t), \mu(t)\} \int_0^{L_i} \tau^d(x,t) \rho(x,t) \mathrm{d}x \right) \nonumber \\
    =& \mathbb{E}^{\bm{\rho}(t)} (\{\bm{w}^d\}^\top(\bm{a} - (\bm{I}-\bm{r})\min\{\bm{c}\bm{s}^*,\bm{\mu}\})),
\end{align}
where $\bm{w}^d \in \mathbb{R}^{|\mathcal{M_{N\cup F}}| \times 1}$ is the column vector of $w^{d}$ for all movements. $\bm{s}^*$ is the signal decision by Transit-MP.
Here we further handle the term $\min\{\bm{c}\bm{s}^*,\bm{\mu}\}$ by adding and subtracting the term $\bm{c}\bm{s}^*$ as below,
\begin{align}
    \eta = \underbrace{\mathbb{E}^{\bm{\rho}(t)}(\{\bm{w}^d\}^\top(\bm{a} - (\bm{I}-\bm{r})\bm{c}\bm{s}^*))}_{\eta_1} + \underbrace{\mathbb{E}^{\bm{\rho}(t)}(\{\bm{w}^d\}^\top(\bm{I}-\bm{r}) (\bm{c}\bm{s}^* - \min\{\bm{c}\bm{s}^*,\bm{\mu}\}))}_{\eta_2}.
\end{align}

\paragraph{Here we handle $\eta_2$.}
Obviously, we have $\bm{0} \preceq \bm{c}\bm{s}^* - \min\{\bm{c}\bm{s}^*,\bm{\mu}\} \preceq \bm{c}\bm{s}^*$ and $\bm{s}^* \preceq 1$. 
Recall that $w_d$ is unbounded for movement in $\mathcal{M_F}$. We assume that when $\mu(t) \leq c$, $\rho(t) \leq \rho^{\mu}$ for movements in $\mathcal{M_F}$. Then, we have
\begin{align}
    w^d(cs(t) - \min\{c s(t), \mu(t)\})
    \begin{cases}
        \leq (cs(t) - \min\{c s(t), \mu(t)\})\rho^{\mu} \leq c \rho^{\mu} \quad \text{if} \quad \mu(t) \leq c \\
         = (cs(t) - cs(t))\rho^{\mu} = 0 \quad \text{if} \quad \mu(t) > c
    \end{cases}
\end{align}
for movements in $\mathcal{M_F}$. 
Then we have, 
\begin{align}
    \eta_2 & \leq \mathbb{E}^{\bm{\rho}(t)}(\{\bm{w}^d\}^\top \bm{I} (\bm{c}\bm{s}^* - \min\{\bm{c}\bm{s}^*,\bm{\mu}\})) \nonumber \\
    &= \mathbb{E}^{\bm{\rho}(t)}(\{\bm{w}^d\}^\top (\bm{c}\bm{s}^* - \min\{\bm{c}\bm{s}^*,\bm{\mu}\})_{\mathcal{M_F}} + \mathbb{E}^{\bm{\rho}(t)}(\{\bm{w}^d\}^\top (\bm{c}\bm{s}^* - \min\{\bm{c}\bm{s}^*,\bm{\mu}\})_{\mathcal{M_N}}\nonumber \\
    & \leq \sum_{\mathcal{M}_{\mathcal{F}}}\mathbb{E}^{\bm{\rho}(t)} (c\rho^{\mu}) + \mathbb{E}^{\bm{\rho}(t)}(\{\bm{w}^d\}^\top \bm{c}\bm{s}^*)_{\mathcal{M_N}} \leq \sum_{\mathcal{M}_{\mathcal{F}}}\mathbb{E}^{\bm{\rho}(t)} (c\rho^{\mu}) + \mathbb{E}^{\bm{\rho}(t)}(\{\bm{w}^d\}^\top \bm{c})_{\mathcal{M_N}} \nonumber \\
    &= \sum_{\mathcal{M}_{\mathcal{F}}}\mathbb{E}^{\bm{\rho}(t)} (c\rho^{\mu}) + \sum_{\mathcal{M}_{\mathcal{N}}}\mathbb{E}^{\bm{\rho}(t)} (c\int_0^{L_i} \tau^d(x,t) \rho(x,t) \mathrm{d}x) \nonumber \\
    & \leq \sum_{\mathcal{M}_{\mathcal{F}}}\mathbb{E}^{\bm{\rho}(t)} (c\rho^{\mu}) + \sum_{\mathcal{M}_{\mathcal{N}}} c \tau^{d,max} L_i \rho^{max} \triangleq K_4
\end{align}

\paragraph{Next, we handle $\eta_1$.}
Note that all vectors and matrices in $\eta_1$ contain fictitious movements with the corresponding size becoming $|\mathcal{M_{N \cup F}}|$. Besides, $\eta_1$ is associated with the signal decision $\bm{s}^*$ by the proposed Transit-MP controller. 
In matrix form, after including movements in $\mathcal{M_F}$, Transit-MP in partially CV environments is written as,
\begin{align}
    \bm{s}^* = \arg\max_{\bm{s}\in \mathcal{S}} (\{\bm{w}^{u,cv}\}^\top - \{\bm{w}^{d,cv}\}^\top \bm{r})\bm{c}\bm{s}. \label{eq: Transit-MP, matrix}
\end{align}
where $\bm{w}^{u,cv}$ and $\bm{w}^{d,cv}$ are the vectors of $w^{u,cv}$ and $w^{d,cv}$ of Transit-MP, respectively. 

Rewriting the pressure calculation of Transit-MP, we have
\begin{align}
    (\{\bm{w}^{u,cv}\}^\top - & \{\bm{w}^{d,cv}\}^\top \bm{r})\bm{c}\bm{s}^* = \max_{\bm{s}\in \mathcal{S}} (\{\bm{w}^{u,cv}\}^\top - \{\bm{w}^{d,cv}\}^\top \bm{r})\bm{c}\bm{s} = \max_{\bm{s}\in \mathcal{S}^{ch}} (\{\bm{w}^{u,cv}\}^\top - \{\bm{w}^{d,cv}\}^\top \bm{r})\bm{c}\bar{\bm{s}} \nonumber \\
    &\geq (\{\bm{w}^{u,cv}\}^\top - \{\bm{w}^{d,cv}\}^\top \bm{r})\bm{c}\bar{\bm{s}} = \{\bm{w}^{d,cv}\}^\top (\bm{I}- \bm{r})\bm{c}\bar{\bm{s}} + (\{\bm{w}^{u,cv}\}^\top - \{\bm{w}^{d,cv}\}^\top)\bm{c}\bar{\bm{s}}, \label{eq: rewrite pressure of transit-mp}
\end{align}

If we assume that the probability of a vehicle being a CV follows a Bernoulli distribution across the network, i.e., $Pr(v \in \mathcal{V}^{cv}_{i,o}) = \pi>0$, then we have
\begin{align}
    \pi \mathbb{E}^{\bm{\rho}(t)}(\bm{w}^d) = \mathbb{E}^{\bm{\rho}(t)}(\bm{w}^{d,cv})
\end{align}
by assuming $\beta_v$ and $\tau_v$ are independent \citep{tan2025cvmp}. 

Then, we have
\begin{align}
    \mathbb{E}^{\bm{\rho}(t)} (\{\bm{w}^{d}\}^\top (\bm{I}- \bm{r})\bm{c}\bm{s}^*) & = \frac{1}{\pi} \mathbb{E}^{\bm{\rho}(t)} (\{\bm{w}^{d,cv}\}^\top (\bm{I}- \bm{r})\bm{c}\bm{s}^*)\nonumber \\
    &= \frac{1}{\pi} \mathbb{E}^{\bm{\rho}(t)}((\{\bm{w}^{u,cv}\}^\top - \{\bm{w}^{d,cv}\}^\top \bm{r})\bm{c}\bm{s}^* - (\{\bm{w}^{u,cv}\}^\top - \{\bm{w}^{d,cv}\}^\top)\bm{c}\bm{s}^*) \nonumber \\
    &\geq \frac{1}{\pi} \mathbb{E}^{\bm{\rho}(t)}(\{\bm{w}^{d,cv}\}^\top (\bm{I}- \bm{r})\bm{c}\bar{\bm{s}} - \underbrace{(\{\bm{w}^{u,cv}\}^\top - \{\bm{w}^{d,cv}\}^\top)\bm{c}(\bm{s}^*-\bar{\bm{s}})}_{\chi_1}).
\end{align}

Back to $\eta_1$, according to Definition \ref{def: admissible}, there exists a positive $\epsilon$ that makes $\bm{a} - (\bm{I} - \bm{r})\bm{c}\bar{\bm{s}} \preceq -\epsilon \bm{1}$, then we have
\begin{align}
     \eta_1 =& \mathbb{E}^{\bm{\rho}(t)}\left(\{\bm{w}^d\}^\top \bm{a} - \{\bm{w}^d\}^\top(\bm{I}-\bm{r}) \bm{c}\bm{s}^*\right) \nonumber \\
     = & \frac{1}{\pi} \mathbb{E}^{\bm{\rho}(t)}\left(\{\bm{w}^{d,cv}\}^\top \bm{a} - \pi\{\bm{w}^d\}^\top(\bm{I}-\bm{r}) \bm{c}\bm{s}^* \right) \nonumber \\
     \leq & \frac{1}{\pi} \mathbb{E}^{\bm{\rho}(t)}\left(\{\bm{w}^{d,cv}\}^\top \bm{a} - \{\bm{w}^{d,cv}\}^\top(\bm{I}-\bm{r}) \bm{c}\bar{\bm{s}} \right) + \frac{1}{\pi} \mathbb{E}^{\bm{\rho}(t)} (\chi_1) \nonumber \\
     = & \frac{1}{\pi} \mathbb{E}^{\bm{\rho}(t)}\left(\{\bm{w}^{d,cv}\}^\top ( \bm{a} - (\bm{I}-\bm{r}) \bm{c}\bar{\bm{s}} )\right) + \frac{1}{\pi} \mathbb{E}^{\bm{\rho}(t)}(\chi_1) \nonumber \\
     \leq & -\frac{1}{\pi} \ \epsilon \ \mathbb{E}^{\bm{\rho}(t)}(\{\bm{w}^{d,cv}\}^\top \bm{1}) + \frac{1}{\pi}\mathbb{E}^{\bm{\rho}(t)}(\chi_1) \nonumber \\
     \leq & -\epsilon \ \mathbb{E}^{\bm{\rho}(t)}(\{\bm{w}^d\}^\top \bm{1}) + \frac{1}{\pi}\mathbb{E}^{\bm{\rho}(t)}(\chi_1)
\end{align}

Note that, if the same weights are used for upstream and downstream movement states for pressure calculation, like Q-MP \citep{varaiya2013max}, D-MP \citep{liu2022novel}, and CV-MP \citep{tan2025cvmp}, $\chi_1 = 0$. This suggests that our stability proof is more generalized.
Specifically, in our cases, 
\begin{align}
    0\leq w^{u,cv}_{i,o}(t) - w^{d,cv}_{i,o}(t) 
    \begin{cases}
        \leq w^{u,cv}_{i,o}(t) \quad (i,o) \in \mathcal{M_N} \\
        =0 \quad (i,o) \in \mathcal{M_F} 
    \end{cases}
\end{align}

Obviously, for the second term of $\eta_1$ we have 
\begin{align}
    \frac{1}{\pi}\mathbb{E}^{\bm{\rho}(t)}(\chi_1) &\leq \frac{1}{\pi}\mathbb{E}^{\bm{\rho}(t)}((\{\bm{w}^{u,cv}\}^\top - \{\bm{w}^{d,cv}\}^\top)\bm{c}\bm{s}^*) \leq \frac{1}{\pi}\mathbb{E}^{\bm{\rho}(t)}((\{\bm{w}^{u,cv}\}^\top - \{\bm{w}^{d,cv}\}^\top)\bm{c}) \nonumber \\
    &\leq \frac{1}{\pi}\sum_{\mathcal{M}_{\mathcal{N}}}\mathbb{E}^{\bm{\rho}(t)} (c\int_0^{L_i} \tau^u(x,t) \rho^{cv}(x,t) \mathrm{d}x)\leq \frac{1}{\pi}\sum_{\mathcal{M}_{\mathcal{N}}} c \tau^{u,max} L_i \rho^{max} \triangleq K_5. \label{eq: K5}
\end{align}

As for the first term of $\eta_1$, recall that $\bm{w}^d$ is the vector of $w^d_{i,o}$ for all $(i,o) \in \mathcal{M_{N \cup F}}$ and 
    \begin{align}
        w^d_{i,o}(t) = 
        \begin{cases}
            \int_0^{L_i} \tau^d_{i,o}(x,t)\rho_{i,o}(x,t) \mathrm{d}x = \bar{\tau}^d_{i,o}(t) z_{i,o}(t)  \quad (i,o) \in \mathcal{M_N}\\
            \rho_{i,o}(t) = z_{i,o}(t) \quad (i,o) \in \mathcal{M_F}
        \end{cases}.
    \end{align}
Here $\bar{\tau}^d_{i,o}(t)$ is a nonnegative constant that must exist. 
Then, there must exist $\tau^{min} = \min\{1, \{\bar{\tau}^d_{i,o}(t)\}_{(i,o) \in \mathcal{M_N},\bar{\tau}^d_{i,o}(t)>0}\}$ that makes
\begin{align}
    \bm{w}^d \succeq \tau^{min}\bm{z}, 
\end{align}
where $\bm{z}$ is the column vector of $z_{i,o}$. Then we have
\begin{align}
    \eta_1 \leq -\epsilon \tau^{min} \mathbb{E}^{\bm{\rho}(t)}(\bm{z}^\top \bm{1}) + K_5 = -\epsilon \tau^{min} \mathbb{E}^{\bm{\rho}(t)} \left(\sum_{(i,o) \in \mathcal{M_{F\cup N}}} z_{(i,o)}(t) \right) + K_5.
\end{align}

In summary, we have
\begin{align}
    \mathbb{E}^{\bm{\rho}(t)} \left(\frac{\mathrm{d} V(\bm{\rho}(t))}{\mathrm{d} t}\right) &= \delta_1 + \delta_2 + \delta_3 \leq K_2 + K_3 + \eta_1 + \eta_2 \nonumber \\
    &\leq K_2 + K_3 +K_4 +K_5 - \epsilon \tau^{min} \mathbb{E}^{\bm{\rho}(t)} \left(\sum_{(i,o) \in \mathcal{M_{F\cup N}}} z_{(i,o)}(t) \right) \nonumber \\
    & \triangleq K - \epsilon'\mathbb{E}^{\bm{\rho}(t)} \left(\sum_{(i,o) \in \mathcal{M_{F\cup N}}} z_{(i,o)}(t) \right),
\end{align}
which proves the stability of the network controlled by Transit-MP in partially CV environments based on Lemma \ref{lemma: sufficient condition}.
\end{proof}

Note that, in Theorem \ref{thm: stability of transit-MP}, the partially CV environments are restricted to scenarios where CV is uniformly distributed, i.e., the penetration rate across the road network is assumed to be identical, which maintained the admissible demand region in Definition \ref{def: admissible}. As demonstrated in \cite{tan2025cvmp}, if CVs are heterogeneously distributed across the network, the admissible demand region of CV-based MP controllers will be reduced. Similar conclusion applies to Transit-MP, as shown in Appendix \ref{appendix: heterogeneously}. 

In summary, we demonstrate that Transit-MP, which uses different weights on upstream and downstream movement states for pressure calculation and prioritizes the high-occupancy vehicles while considering the impact of transit stations, can stabilize the road network queue in partially CV environments. 

\section{Modified Transit-MP for sparse CV environments}
Theorem \ref{thm: necessary condition} indicates that CV-MP may fail to stabilize the network queues in the case when no CVs are observed in the movements of a phase until spillovers. In fact, the theorem applies to all MP controllers based only on real-time CV data, including the proposed Transit-MP. In this section, we propose a modified Transit-MP, denoted by mTransit-MP, to address this issue. For those movement where real-time CV data is not available, we will incorporate historical CV data to provide pressure estimates to modify Transit-MP, thus avoiding the unfavorable situation in Theorem \ref{thm: necessary condition}. 

\subsection{Estimated movement travel time}
When there are no CV observations, the expected arrival rate is first estimated by historical CV data, which has been extensively studied by existing research \citep{zheng2017estimating, tang2020tensor, tan2025robust}. We will therefore not dive into CV-based arrival rate estimates in this study. 
Since the CV data is sparse, the estimated arrival rate is essentially an average value over a time-of-day period for collecting multiple-day CV data. Considering the time-varying nature of traffic flow throughout the day, we can estimate an arrival rate for each period, e.g., 15 min or longer, that depends on the available CV data, to accommodate traffic flow variations.
Besides, in this study, we assume that only CV data is available for MP control with TSP. In reality, if there are fixed-location detectors such as loop detectors deployed on the link, the expected arrival rate can be directly obtained by detector data. 

With the estimated arrival rate, the corresponding CV penetration rate $\hat{\psi}_{i,o}$ can be further obtained
\begin{align}
    \hat{\psi}_{i,o} = \frac{|\mathcal{V}^{hcv}_{i,o}|}{\hat{\lambda}_{i,o} T_{tod}}
\end{align}
where $T_{tod}$ denotes the time-of-day period. In some studies \citep{wong2019estimation, jia2023uncertainty}, they may use CV data to estimate the penetration rate first, and then estimate the arrival rate or traffic volume based on the penetration rate, which applies in the following steps of mTransit-MP as well. 

Besides, existing studies have demonstrated that, with even a single CV observed in the queue, the real-time queue length can be estimated \citep{tan2019cycle, li2017real}. Combined with upstream signal information, such a queue length estimate can be more accurate \citep{yang2018queue}. Since this is not the focus of this paper, we assume that for those decision steps with CV observation, the queue length is estimated based on CVs. 

Given the expected arrival rate $\hat{\lambda}_{i,o}$ and the corresponding penetration rate $\hat{\psi}_{i,o}$, we can have an estimate of the movement travel time even though no CVs are observed on the link. Based on the incremental queue accumulation (IQA) model \citep{strong2006new}, the real-time queue length at the stopline is calculated as 
\begin{align}
    Q_{i,o}(t+1) = Q_{i,o}(t) + A_{i,o}(t) - D_{i,o}(t), \label{eq: real-time queue}
\end{align}
where $Q_{i,o}$ denotes the accumulated vehicle at the stopline neglecting the physical length, $A_{i,o}$ denotes the arrived vehicle, and $D_{i,o}$ denotes the departed vehicle. $A_{i,o}(t)$ is determined by the real-time vehicle arrivals and $D_{i,o}(t)$ depends on the signal state. 

Recall that $T_0$ denotes the decision step length of the MP controller, then we have
\begin{align}
    Q_{i,o}(t) = Q_{i,o}(t-T_0) + \int_{t-T_0}^{t} A_{i,o}(t) \mathrm{d}t - \int_{t-T_0}^{t} D_{i,o}(t)\mathrm{d}t,
\end{align}

In expectation, we have 
\begin{align}
    \mathbb{E}(Q_{i,o}(t)) = \max\{0, \mathbb{E}(Q_{i,o}(t-T_0)) + \hat{\lambda}_{i,o} T_0 - s_{i,o}(t-T_0)\lambda_{i,o}^{depart} T_0\}. \label{eq: expected Q}
\end{align}
where $\lambda_{i,o}^{depart}$ denotes the expected departure rate at the stopline. In particular, for those moments with CVs, the queue length estimated by CVs can be used as a substitute for $Q_{i,o}(t-T_0)$, which corrects the queue length estimates to avoid accumulated errors due to the error of $\hat{\lambda}_{i,o}$ and $\lambda_{i,o}^{depart}$.

By using the IQA model, the spatial distribution of vehicles is ignored; thus, only the travel time of accumulated queuing vehicles is considered. For a vehicle $v$, its link travel time is approximated as
\begin{align}
    LTT_v(t) = \overline{ETT}_{i,o} + Delay_v(t),
\end{align}
where $Delay_v$ is the delay of vehicle $v$ by IQA model. Recall that $\overline{ETT}_{i,o}$ is the expected no-delay travel time. Then, the total travel time of the movement, denoted by $TTT_{i,o}$, is calculated as
\begin{align}
    TTT_{i,o}(t) = \sum_{v \in \mathcal{Q}_{i,o}(t)} LTT_v(t) = \overline{ETT}_{i,o} Q_{i,o}(t) + \sum_{v \in \mathcal{Q}_{i,o}(t)} Delay_v(t),
\end{align}
where $\mathcal{Q}_{i,o}$ denotes the set of queued vehicles calculated by the IQA model (corresponding to $Q_{i,o}$).

\begin{figure}[ht]
    \centering
    \includegraphics[width=0.7\textwidth]{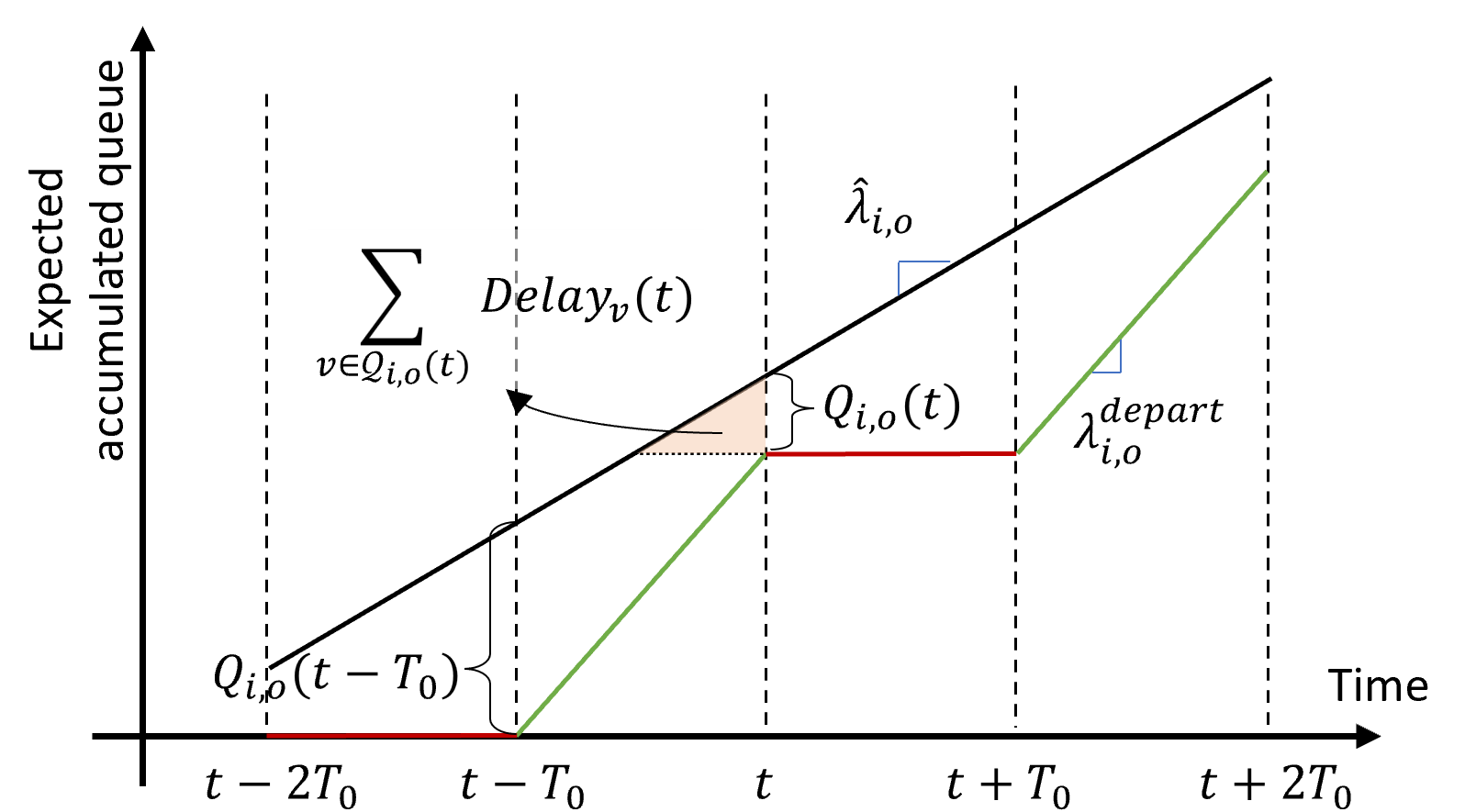}
    \caption{Estimated total delay of vehicles on the link when no CVs are observed.}
    \label{fig: IQA delay}  
\end{figure}

The calculation of $\sum_{v \in \mathcal{Q}_{i,o}(t)} Delay_v(t)$ is illustrated in Fig. \ref{fig: IQA delay}. Note that since we are calculating the total travel time of vehicles on the link, the delay of those departure vehicles should be excluded. Rather than summing the queue length over the period $T_0$, $\sum_{v \in \mathcal{Q}_{i,o}(t)} Delay_v(t)$ is calculated as the triangle area that corresponds to the delay of vehicles in $\mathcal{Q}_{i,o}(t)$. In expectation, we have 
\begin{align}
    \mathbb{E}\left(\sum_{v \in \mathcal{Q}_{i,o}(t)} Delay_v(t)\right) = \frac{\mathbb{E}(Q_{i,o}(t))^2}{2 \hat{\lambda}_{i,o}}.
\end{align}

Normalized by expected no-delay travel time $\overline{ETT}_{i,o}$, the final movement state for pressure calculation, denoted by $\hat{\tau}_{i,o}$ is estimated as
\begin{align}
    \hat{\tau}_{i,o} = \frac{\hat{\psi}_{i,o} \mathbb{E}(TTT_{i,o}(t))}{\overline{ETT}_{i,o}} = \hat{\psi}_{i,o}\mathbb{E}(Q_{i,o}(t)) + \frac{\hat{\psi}_{i,o} \mathbb{E}(Q_{i,o}(t))^2}{2 \hat{\lambda}_{i,o} \overline{ETT}_{i,o}}, \label{eq: estimated TT}
\end{align}
where $\mathbb{E}(Q_{i,o}(t))$ is estimated by Eq.~\eqref{eq: expected Q} at each signal decision step. Recall that $\hat{\psi}_{i,o}$ is the estimated penetration rate of the movement. $\hat{\tau}_{i,o}$ actually represents the normalized link travel time in the context of partially CV environments when no CVs are observed on the movement. By multiplying the penetration rate, it can be ensured that all movement pressures are calculated in the context of partially CV environments.

\subsection{mTransit-MP}
Given the estimated movement state $\hat{\tau}_{i,o}$, the mTransit-MP in sparse CV environments is applied by:
\begin{align}
    \bm{s}^* &= \arg \max_{\bm{s} \in \mathcal{S}} \sum_{n \in \mathcal{N}} \left( \sum_{\forall( i, o) \in \mathcal{M}_{n}} s_{i,o} c_{i,o} \left( \tau_{i,o}^{m,p} - \sum_{ (o,\forall k)\in \mathcal{M}_{n'}} r_{o,k} \tau_{o,k} \right)\right), \label{eq: m-transit MP}
\end{align}
where $\tau_{i,o}^{m,p}$ is the modified upstream travel time incorporating occupancy information, $\tau_{o,k}$ is the downstream travel time, and
\begin{align}
    &\tau_{i,o}^{m,p} = \begin{cases} 
        \hat{p}_{i,o}\hat{\tau}_{i,o} \quad &\text{if} \quad \mathcal{V}_{i,o}^{cv} = \varnothing \\
        \sum_{v \in \mathcal{V}_{i,o}^{cv}} \beta_v p_v \tau_v \quad &\text{else} 
    \end{cases},\label{eq: modified upstream tt}\\
    &\tau_{o,k} = \sum_{v \in \mathcal{V}_{o,k}^{cv}} \beta_v\tau_v. \label{eq: modified downstream tt}
\end{align}
In which $\hat{p}_{i,o}$ is the average occupancy of the movement that can be obtained from historical CV data.
Eq.~\eqref{eq: modified upstream tt} can be interpreted as that if there are no CVs observed, the movement travel time estimated by historical CVs is used as an alternative; otherwise, it uses the real-time CV data. In particular, similar to Transit-MP, the upstream average occupancy information $\hat{p}_{i,o}$, which is estimated using historical data, is incorporated to prioritize high-occupancy movements in mTransit-MP. 

\subsection{Queue starvation immunity of mTransit-MP}
Here we define the phenomenon of queue starvation as below:
\begin{definition} [Queue starvation] \label{def: queue starvation}
    A movement $(i,o) \in \mathcal{M_N}$ is said to experience queue starvation if, there exists a moment $t_0$ with the number of vehicles $z_{i,o}(t_0) > 0$, for any $t \geq t_0$, the signal state $s_{i,o}(t) = 0$. 
\end{definition}
Definition \ref{def: queue starvation} describes a phenomenon where a queue exists for the movement, but the movement can no longer receive a green phase. As shown in Figure \ref{fig: queue starvation}, in sparse CV scenarios, this phenomenon could happen if the movement is spillover without any CV observations, thus violating the sufficient condition of CV observations in Theorem \ref{thm: necessary condition}. 
The proposed Transit-MP, similar to CV-MP, which only uses real-time CV data for pressure calculation, may also experience queue starvation in sparse CV environments. 

\begin{figure}[ht]
    \centering
    \includegraphics[width=0.9\textwidth]{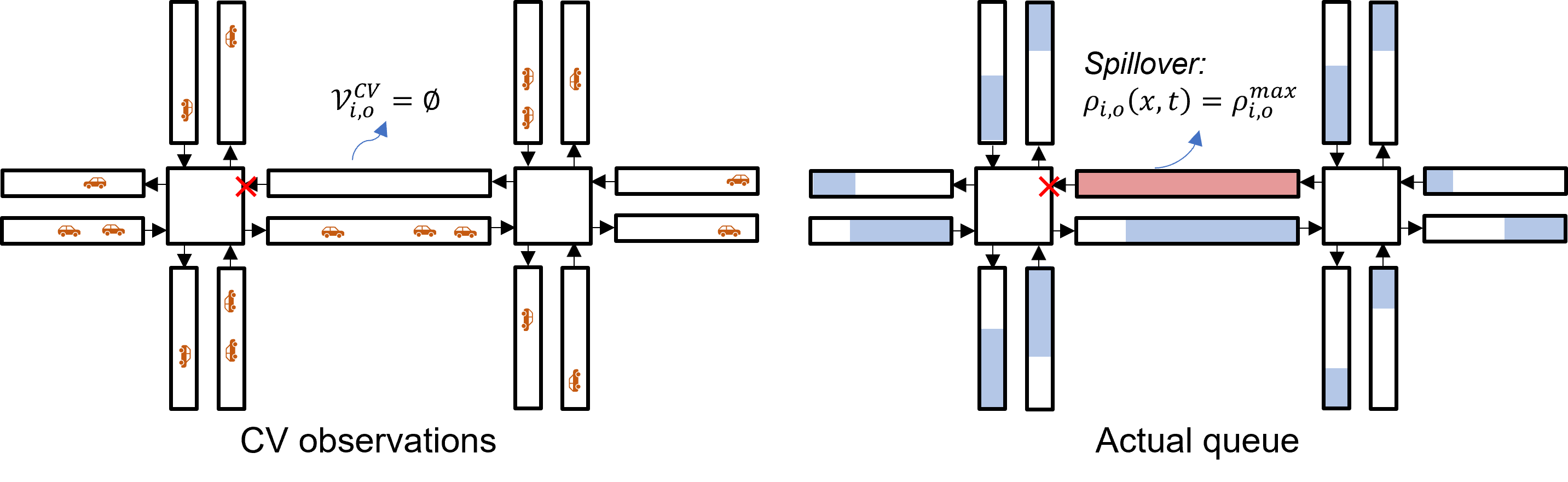}
    \caption{Queue starvation phenomenon due to lack of CV observations.}
    \label{fig: queue starvation}  
\end{figure}

According to Theorem \ref{thm: queue starvation immunity}, after incorporating historical CV data for the modification of pressure calculation, the proposed mTransit-MP can effectively avoid the unfavorable case of queue starvation that will unstabilize the network queue. 

\begin{theorem} [Queue starvation immunity of mTransit-MP] \label{thm: queue starvation immunity}
    Given the admissible demand region $\bm{\Lambda}$ in Definition \ref{def: admissible}, if the exogenous demand $\bm{a} \in \bm{\Lambda}^{int}$, the probability of a network controlled by mTransit-MP presented in Eq.~\eqref{eq: m-transit MP} experiencing queue starvation is 0 even in sparse CV environments, indicating that mTransit-MP is queue starvation immunity. 
\end{theorem}

\begin{proof}
    Definition \ref{def: queue starvation} is equivalent to saying that if a movement $(i,o)$ experiences queue starvation, in the long term, the average signal state $\bar{s}_{i,o}$ of the movement asymptotes to 0, i.e.,
    \begin{align}
        \bar{s}_{i,o} = \lim_{T \rightarrow \infty} \frac{1}{T} \sum_{t=1}^T s_{i,o} = \lim_{T \rightarrow \infty} \frac{1}{T} (\underbrace{\sum_{t=1}^{t_0} s_{i,o}}_{=\bar{s}_{i,o \mid \leq t_0}} + \underbrace{\sum_{t=t_0+1}^T s_{i,o}}_{=0}) = \lim_{T \rightarrow \infty} \frac{1}{T}  \bar{s}_{i,o \mid \leq t_0} = 0, \label{eq: t_0 exists}
    \end{align}
    where $\bar{s}_{i,o \mid <t_0}$ is a constant and $\bar{s}_{i,o \mid \leq t_0} \leq t_0$. 

    Then, Theorem \ref{thm: queue starvation immunity} is equivalent to saying that for $\forall (i,o) \in \mathcal{M_N}$ in the network controlled by mTransit-MP, the average signal state $\bar{s}_{i,o}$ asymptotes to a positive constant, i.e., 
    \begin{align}
        \bar{s}_{i,o} = \lim_{T \rightarrow \infty} \frac{1}{T} \sum_{t=1}^T s_{i,o}  > 0.
    \end{align}

    In other words, we need to prove that such a $t_0$ in Eq.~\eqref{eq: t_0 exists} does not exist for any movement $(i,o) \in \mathcal{M_N}$ under the control of mTransit-MP. 
    
    This can be easily proved by contradiction. Suppose that there exists a $t_0$ such that for any $t \geq t_0$, the signal state $s_{i,o}(t) = 0$. Then, looking into the modified traffic state of mTransit-MP presented in Eq.~\eqref{eq: modified upstream tt}, we have
    \begin{align}
        \tau^{m,p}_{i,o}(t) =
        \begin{cases}
            \sum_{v \in \mathcal{V}_{i,o}^{cv}} \beta_v (t) p_v (t) \tau_v (t) \quad \text{if} \quad \mathcal{V}_{i,o}^{cv} \neq \varnothing\\
            \hat{p}_{i,o}\hat{\tau}_{i,o}(t) \quad \text{if} \quad \mathcal{V}_{i,o}^{cv} = \varnothing
        \end{cases}.
    \end{align}
         
    In the case of $\mathcal{V}_{i,o}^{cv} \neq \varnothing$, as the signal state $s_{i,o}(t) = 0$ always holds, we have $\mathcal{V}_{i,o}^{cv} (t_0) \subseteq \mathcal{V}_{i,o}^{cv} (t)$, where $\mathcal{V}_{i,o}^{cv} (t_0)$ denotes those CVs on the link at moment $t_0$.
    Obviously, we have
    \begin{align}
        \tau^{m,p}_{i,o}(t) = \sum_{v \in \mathcal{V}_{i,o}^{cv} (t)} \beta_v (t) p_v (t) \tau_v (t) \geq \sum_{v \in \mathcal{V}_{i,o}^{cv} (t_0)} \beta_v (t) p_v (t) \tau_v (t), 
    \end{align}
    
    Note that, if the signal state is always 0 during the period $[t_o, t]$, then the link travel time of a CV $v$ at moment $t$ is calculated as $\tau_v(t) = \tau_v(t_0)+\frac{t-t_0}{\overline{ETT}_{i,o}}$. As 
    \begin{align}
        \lim_{t \rightarrow \infty} \sum_{v \in \mathcal{V}_{i,o}^{cv} (t_0)} \beta_v (t) p_v (t) \tau_v (t) = \lim_{t \rightarrow \infty} \sum_{v \in \mathcal{V}_{i,o}^{cv} (t_0)} \beta_v (t) p_v (t) (\tau_v (t_0)+ \underbrace{\frac{t-t_0}{\overline{ETT}_{i,o}}}_{\rightarrow \infty}) \rightarrow \infty, 
    \end{align}
    we have $\lim_{t \rightarrow \infty}\tau^{m,p}_{i,o}(t) \rightarrow \infty$ in the case of $\mathcal{V}_{i,o}^{cv} \neq \varnothing$.

    Regarding the case of $\mathcal{V}_{i,o}^{cv} = \varnothing$, i.e., no CV is observed, as $s_{i,o}(t) = 0$ for $t \geq t_0$, we have 
    \begin{align}
        \mathbb{E}(Q_{i,o}(t)) = \mathbb{E}(Q_{i,o}(t_0) + \int_{t_0}^tA_{i,o}(t)\mathrm{d}t) = Q_{i,o}(t_0) + (t-t_0)\hat{\lambda}_{i,o},
    \end{align}
    which gives 
    \begin{align}
        \hat{\tau}_{i,o}(t) = (Q_{i,o}(t_0) + (t-t_0)\hat{\lambda}_{i,o})\hat{\psi}_{i,o} + \frac{(Q_{i,o}(t_0) + (t-t_0)\hat{\lambda}_{i,o})^2\hat{\psi}_{i,o}}{2 \hat{\lambda}_{i,o} \overline{ETT}_{i,o}}
    \end{align}
    based on Eq.~\eqref{eq: estimated TT}. Obviously, we have
    \begin{align}
        \lim_{t \rightarrow \infty} \hat{p}_{i,o}\hat{\tau}_{i,o}(t) \rightarrow \infty
    \end{align}
    
    In summary, in both cases, we have $\lim_{t \rightarrow \infty}\tau^{m,p}_{i,o}(t) \rightarrow \infty$. Then, there must exist a moment $t' \geq t_0$ that makes the pressure of movement $(i,o)$ exceeding all other movements at the intersection, leading to $s_{i,o}(t') = 1$, which is a contradiction to the supposition that for any $t \geq t_0$, the signal state $s_{i,o}(t) = 0$. Thereby, such a $t_0$ in Eq.~\eqref{eq: t_0 exists} does not exist for any movement $(i,o) \in \mathcal{M_N}$ under the control of mTransit-MP, which completes the proof of Theorem \ref{thm: queue starvation immunity}.
    
\end{proof}

\section{Evaluation Results}
\subsection{Experimental settings}
A multi-modal corridor with three signalized intersections in Amsterdam is simulated by SUMO to evaluate the performance of the proposed MP controllers under realistic transit operation scenarios. Within the area, 7 tram lines and 8 bus lines are operated between 31 transit stations, which are strategically located near intersections or mid-links. The lane and signal phase configuration of three intersections is presented in Figure \ref{fig: studied corridor}. Three types of lanes are included: dedicated lanes for transit vehicles, mixed-use lanes for both private cars and transit vehicles, and general lanes for private cars. Movements within a box, indicated by arrows, constitute a phase, with arrow colors corresponding to different types of lanes. Specifically, within the same phase, trams have the highest priority, followed by buses, with private vehicles yielding last. In the Netherlands, transit vehicles emit a warning chime in practice when passing through the intersection, and private vehicles must yield. 
The MP controller can only activate the corresponding phase at each decision step.

The departure intervals for public transit lines range from 10 minutes to 20 minutes, which are obtained from the published time schedule. To model the occupancy dynamics of transit vehicles, passengers with random OD demands are generated to simulate the boarding and alighting process. Similar to \cite{tan2025cvmp}, during the three-hour simulation, the background traffic is input based on OD pairs, where the two-way main road OD demands (1-5 and 5-1) experience an increasing process and then decrease to model the time-varying traffic and test the capability of different controllers. 
Each experiment is repeated three times with different random seeds. 

\begin{figure}[ht]
    \centering
    \includegraphics[width=0.9\textwidth]{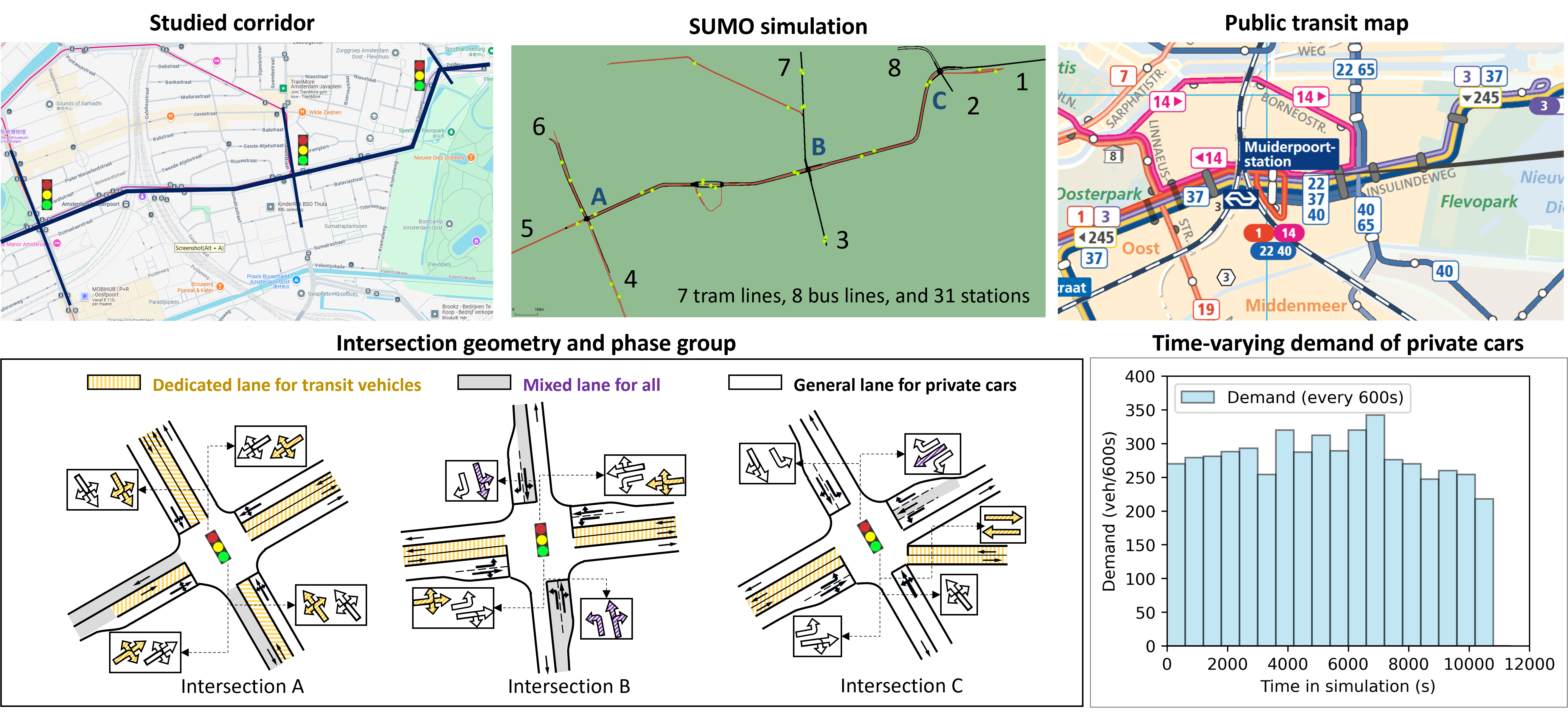}
    \caption{Real-world corridor with multi-modal traffic.}
    \label{fig: studied corridor}  
\end{figure}

The following methods are evaluated:
\begin{itemize}
    \item \textbf{OCC-MP} \citep{ahmed2024occ}: OCC-MP uses the average upstream occupancy $\overline{p}_{i,o}$ to weight the final movement pressure that is calculated based on the queue length (essentially the number of vehicles). To reduce the impact of various link lengths\footnote{\cite{varaiya2013max} offers two approaches to address long-link issues: link-length weighting and link segmentation into equal lengths. This paper primarily employs the former for experimentation. Since this is not the focus of this study and to alleviate reader concerns, we supplement the segmentation-based results in Appendix \ref{appendix: seg results}.}, the queue length is actually weighted by the reverse of the square root of the link length \citep{varaiya2013max}. The impact of transit stations is not considered. Specifically, the control policy of OCC-MP is 
    \begin{align}
        \bm{s}^* &= \arg \max_{\bm{s} \in \mathcal{S}} \sum_{n \in \mathcal{N}} \left( \sum_{\forall( i, o) \in \mathcal{M}_{n}} s_{i,o} c_{i,o} \overline{p}_{i,o}\left(\frac{\sum_{v \in \mathcal{V}_{i,o}^{cv}} 1}{\sqrt{L_i}} - \sum_{ (o,\forall k)\in \mathcal{M}_{n'}} r_{o,k} \frac{\sum_{v \in \mathcal{V}_{o,k}^{cv}} 1}{\sqrt{L_o}} \right)\right) \label{eq: OCC MP} 
    \end{align}
    where only CV data is used in partially CV environments. 
    \item \textbf{eOCC-MP}: eOCC-MP is the extended version of OCC-MP, which further considers the impact of transit stations by introducing parameter $\beta_v$ presented in Eq.~\eqref{eq: beta}, i.e., 
    \begin{align}
        \bm{s}^* &= \arg \max_{\bm{s} \in \mathcal{S}} \sum_{n \in \mathcal{N}} \left( \sum_{\forall( i, o) \in \mathcal{M}_{n}} s_{i,o} c_{i,o} \overline{p}_{i,o}\left(\frac{\sum_{v \in \mathcal{V}_{i,o}^{cv}} \beta_v}{\sqrt{L_i}} - \sum_{ (o,\forall k)\in \mathcal{M}_{n'}} r_{o,k} \frac{\sum_{v \in \mathcal{V}_{o,k}^{cv}} \beta_v}{\sqrt{L_o}} \right)\right) \label{eq: eOCC MP} 
    \end{align}
    \item \textbf{Transit-MP}: Transit-MP is the proposed extension of CV-MP controller for transit signal priority, which further incorporates the real-time vehicle occupancy at upstream links and considers the impact of transit stations, i.e., Eq.~\eqref{eq: transit MP}. 
    \item \textbf{mTransit-MP}: mTransit-MP is the proposed modified version of Transit-MP that is specialized for sparse CV environments, which can mitigate the queue starvation phenomenon by incorporating historical CV data, i.e., Eq.~\eqref{eq: m-transit MP}. 
\end{itemize}

The decision step $T_0$ is set to 10 seconds for all MP controllers, with the phase transition period consisting of $T_y = 3$ seconds of yellow time and no red clearance time. Considering the lost time due to the phase transition period, the saturated flow rate discount is considered by multiplying $(T_0-T_y-T_l)/T_0$ when switching phases, where $T_l=1$ is the green start-up lost time.

The following metrics are used to evaluate the performance of different MP controllers:
\begin{itemize}
    \item \textit{Average vehicle delay} (s/veh): the average time loss of vehicles during the network. According to vehicle types, we also have CV delay (private car only), NV delay, and transit delay.
    \item \textit{Average passenger delay} (s/veh): the average control delay of passengers taking transit vehicles. 
    \item \textit{Vehicle count} (veh): The real-time number of vehicles on the network. With the same traffic load, a lower vehicle count indicates a higher operational efficiency. We can use the maximum vehicle count to evaluate the overall performance of the controller during the whole simulation. 
    \item \textit{Spillover count} (veh): The real-time number of vehicles that are blocked from loading due to spillover at source links. Continuous increases in spillover count indicate that the network is unstable as the demand exceeds the capacity. Similarly, we can also use the maximum spillover count to evaluate the overall performance of the controller during the whole simulation.
    \item \textit{Unserved count} (veh): If we define the served vehicle as a vehicle that left the network, then the unserved vehicle is defined as the sum of the vehicle count and the spillover count. The lower the unserved count, the higher the throughput capacity of the controller, as more vehicles are served given the same traffic load. The maximum unserved count can also be used to evaluate the overall performance of the controller during the whole simulation.
\end{itemize}

\subsection{Performance in fully CV environments} \label{ssec: fully CV}
We first test the ideal performance of three MP controllers, i.e., OCC-MP, eOCC-MP, and Transit-MP, in a fully CV environment. Note that in such an environment, mTransit-MP is exactly the same as Transit-MP.

Figure \ref{fig: overall} shows that factoring transit-station effects into the controller (eOCC-MP) markedly improves network performance relative to the original OCC-MP. Under the control of eOCC-MP, transit vehicle delay falls by 23–30\%, yielding a 31.6\% reduction in passenger delay, while private car delay also drops by 16.9\%. Throughput performance also shows improvement. The peaks in vehicle count, spillover count, and unserved count are all noticeably lower, indicating better throughput capability once the impact of transit stations is captured. Building on those gains, the proposed transit-prioritized controller, Transit-MP, delivers still stronger results. Compared to eOCC-MP, it cuts transit-vehicle delay by a further 5–43\%, trimming passenger delay by an additional 17.9\%, and lowering private-car delay by another 21.8\%. Throughput benefits are even more pronounced by Transit-MP. The maximum spillover shrinks to 5.3 vehicles, which is 94.2\% less than eOCC-MP, while the maximum vehicle count and the maximum unserved count also decline sharply. The above results demonstrate the superiority of the proposed Transit-MP in smoothing the operation of multi-model traffic networks while maximizing the network throughput with mitigated spillovers. 

\begin{figure}[ht]
    \centering
    \includegraphics[width=0.9\textwidth]{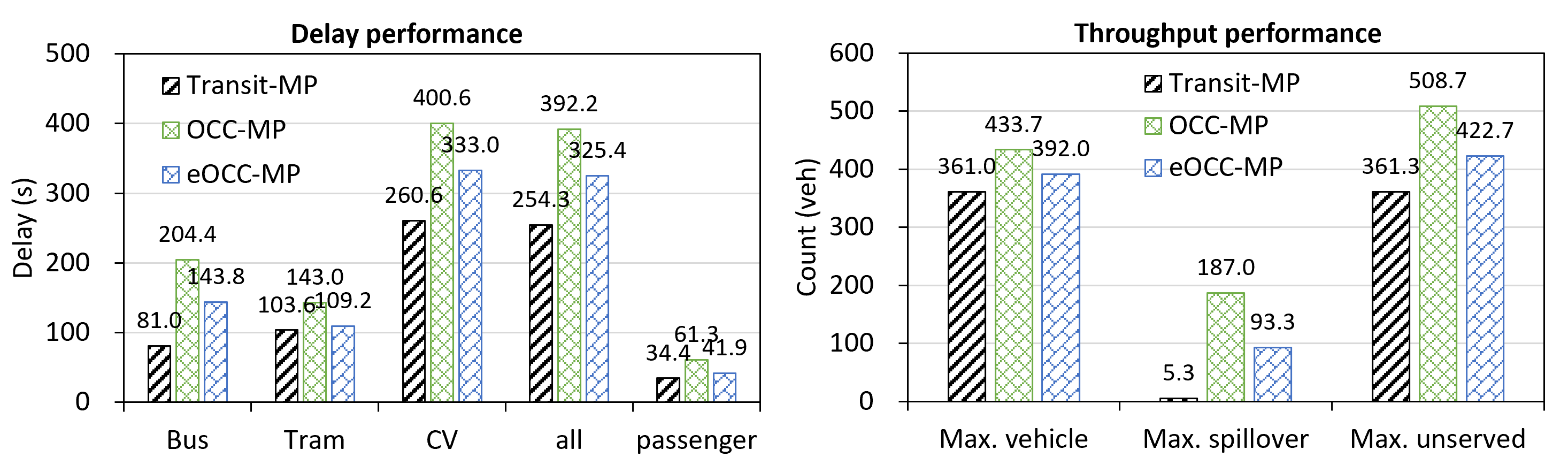}
    \caption{Overall performance of three MP controllers in fully CV environments.}
    \label{fig: overall}  
\end{figure}

\begin{figure}[ht]
    \centering
    \includegraphics[width=0.9\textwidth]{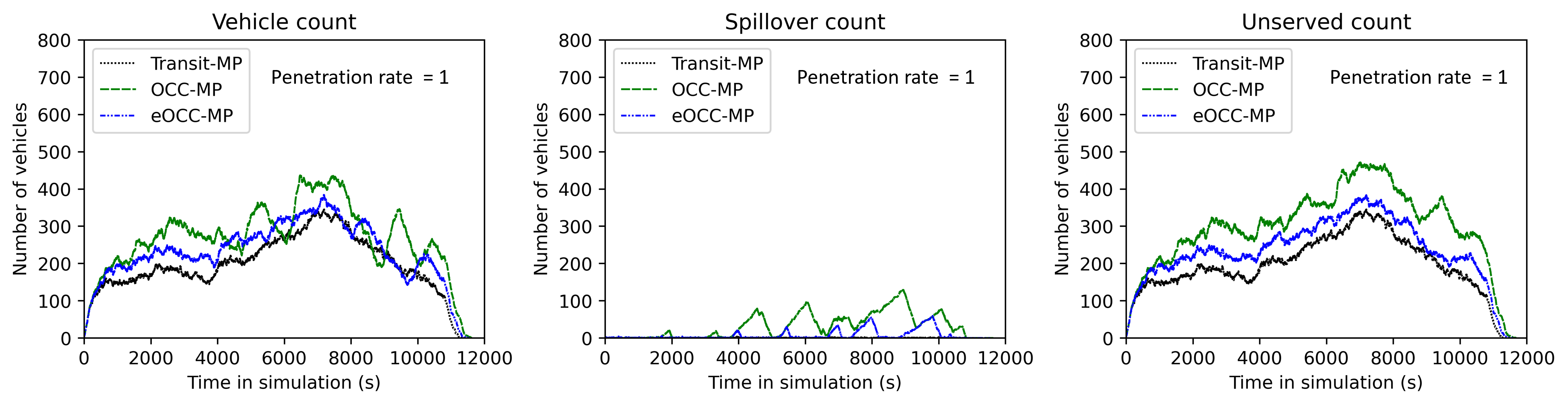}
    \caption{Real-time throughput  performance of three MP controllers in fully CV environments.}
    \label{fig: all cv real time}  
\end{figure}

Figure \ref{fig: all cv real time} presents the real-time throughput performance throughout the simulation. The vehicle count panel shows that the traffic demand builds rapidly during the first 2000~s and peaks between 6000~s and 8000~s. Among three MP controllers, Transit-MP keeps the number of vehicles in the network lowest, topping out at roughly 330 vehicles, while eOCC-MP peaks at around 380 vehicles, and OCC-MP climbs to about 450 vehicles. Once demand subsides after 10000 s, Transit-MP also clears the residual vehicles in the network fastest. A similar trend appears in the unserved count panel, indicating that Transit-MP can better accommodate traffic under the same traffic loading compared to OCC-MP and eOCC-MP. The difference is most striking in the spillover count pane. Transit-MP almost eliminates spillovers, whereas eOCC-MP reaches modest spikes up to 80 vehicles, and OCC-MP suffers the largest overflows, brief surges exceeding 120 vehicles around 9000~s. Taken together, the three panels confirm that factoring station effects in MP control can improve multi-modal traffic performance, and the proposed Transit-MP yields the most stable and efficient real-time performance, achieving lower in-network vehicles, minimal spillover, and faster recovery once the peak has passed.

The primary reason for the poor performance of OCC-MP in our cases lies in its practice of incorporating vehicles into pressure calculations immediately upon their entry into the link. When dealing with longer links, as exemplified in this study, prematurely factoring in vehicles with high occupancy without considering their distance from intersections results in granting priority green lights to transit vehicles too early. This leads to significant green light wastage, thereby increasing delays and spillovers for vehicles in other movements. Additionally, when transit vehicles require station stops along the link, the persistent priority assignment of OCC-MP similarly causes green light wastage. For issues arising from excessively long links, a feasible improvement involves segmenting the link \citep{varaiya2013max, ahmed2024occ}, though the optimal segmentation method remains open for discussion. eOCC-MP addresses the second issue by accounting for station impacts, effectively preventing priority green light wastage during transit vehicle stops at stations and achieving better performance than OCC-MP. Transit-MP, by using vehicle travel time as the basis for calculating pressure, accounts for distance from intersections. This allows it to better balance transit priority with private vehicle delays, making it less susceptible to the effects of long links. Furthermore, by also considering the impact of station stops, Transit-MP allocates green light time more reasonably, achieving the best overall performance.

\subsection{Performance in partially CV environments}
Figure \ref{fig: different p} contrasts four controllers, i.e., OCC-MP, eOCC-MP, Transit-MP, and its modified version mTransit-MP, over rising CV penetration rates from 10\% to 70\%. Note that, since we do not specify the arrival rate estimator for mTransit-MP, we use the average 30-minute flow rate of the first experiment as the arrival rate estimate. In this case, the average arrival rate for the second and third experiments is actually with errors, as the input demand varies given different random seeds. As for the real-time queue length in Eq.~\eqref{eq: real-time queue}, we directly extract it from the simulation. We do this in order to first test the performance of mTransit-MP in the ideal case, and in subsequent sections, we will further test the impact of parameter estimation errors.

\begin{figure}[ht]
    \centering
    \includegraphics[width=0.9\textwidth]{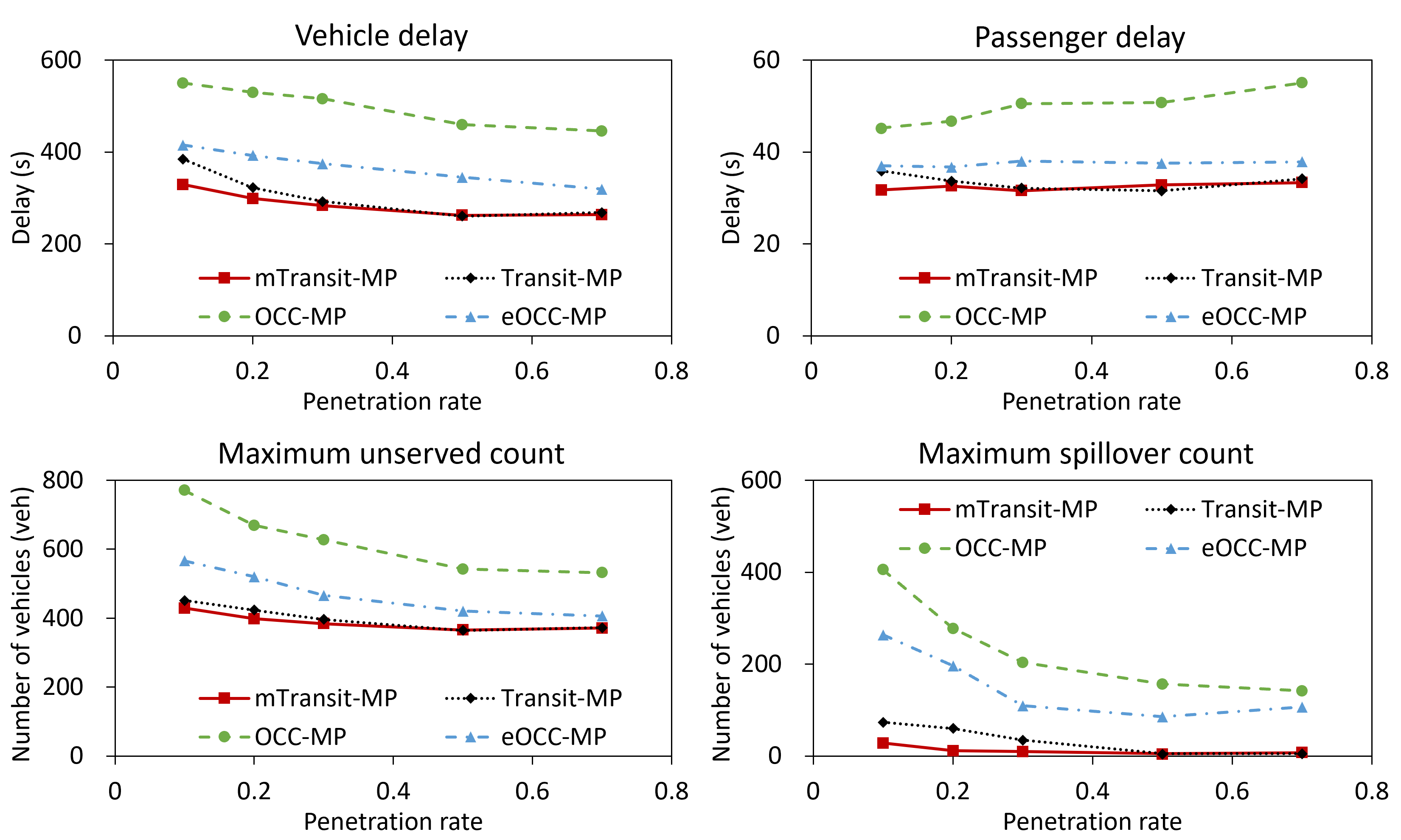}
    \caption{Overall performance of four MP controllers at different penetration rates.}
    \label{fig: different p}  
\end{figure}

Overall, as CV penetration rate rises, all network performance metrics, i.e., average vehicle delay, maximum unserved vehicle count, and maximum spillover count, decline almost monotonically. This is because, with more CV data, the pressures fed to the MP-based controllers better reflect actual demand, allowing them to respond more effectively. Passenger delay, the key transit performance metric, moves in the opposite direction. For mTransit-MP, OCC-MP, and eOCC-MP, it increases as the penetration rate grows. This is because private vehicles dominate the pressure calculation at higher penetration rate levels; transit vehicles receive proportionally less weight and thus enjoy less priority. Note that passenger delay results of Transit-MP show a different trend in sparse CV environments when the penetration rate is no more than 0.2; it decreases when the penetration rate grows to 0.3. This is due to the fact that in some of the experiments in sparse CV scenarios, significant spillover occurs at some links due to the more extreme CV distributions, resulting in an overall degradation of the road network performance. OCC-MP and eOCC-MP experience spillover in all penetration rate scenarios, so the trend is consistent; mTansitMP almost mitigates spillover in sparse CV scenarios, so the trend is also consistent.

Across all penetration rates, the same performance hierarchy holds. OCC-MP consistently performs the worst. Factoring transit station effects in pressure calculation helps eOCC-MP narrow the gap, improving both network and transit performance. Even so, the two proposed MP controllers, Transit-MP and mTransit-MP, outperform the others on every measure and are notably less sensitive to penetration rate changes, which confirms the superiority of the proposed method.

When the penetration rate reaches 0.5, the performance difference between Transit-MP and mTransit-MP virtually disappears. At this point, it is rare for a short link to have no observed CVs, so Transit-MP almost has no spillover, and the modification mechanism unique to mTransit-MP is seldom triggered. As shown in Figure \ref{fig: 0.5 penetration}, at 0.5 penetration rates, the real-time performance of Transit-MP and mTransit-MP has no significant difference. 

\begin{figure}[ht]
    \centering
    \includegraphics[width=0.9\textwidth]{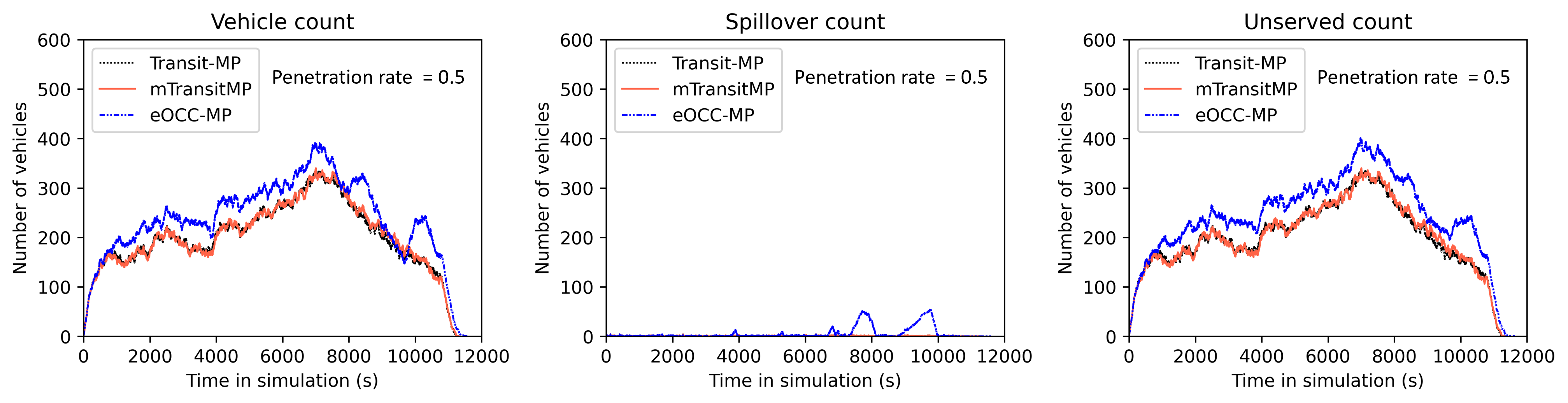}
    \caption{Real-time performance of mTransit-MP at a 0.5 penetration rate.}
    \label{fig: 0.5 penetration}  
\end{figure}

In sparse CV environments (0.1 and 0.2 penetration rates), however, short links can suffer spillovers when no CVs are observed for Transit-MP and eOCC-MP, sharply increasing the vehicle delay, unserved count, and spillover count compared to mTransit-MP. 
By leveraging historical traffic data, mTransit-MP prevents these spillovers, maintaining network performance under low penetration rate conditions.

\subsection{Detailed performance in sparse CV environments}
Figure \ref{fig: 0.1 penetration} compares three MP controllers, eOCC-MP, Transit-MP, and mTransit-MP, under a sparse CV environment (0.1 penetration). In terms of vehicle count, the three controllers show insignificant differences for most of the time. Overall, mTransit-MP maintains the lowest count, Transit-MP is slightly higher, and eOCC-MP performs the worst. Nevertheless, after roughly 9000 seconds, the vehicle count of eOCC-MP rises sharply. 

\begin{figure}[ht]
    \centering
    \includegraphics[width=0.9\textwidth]{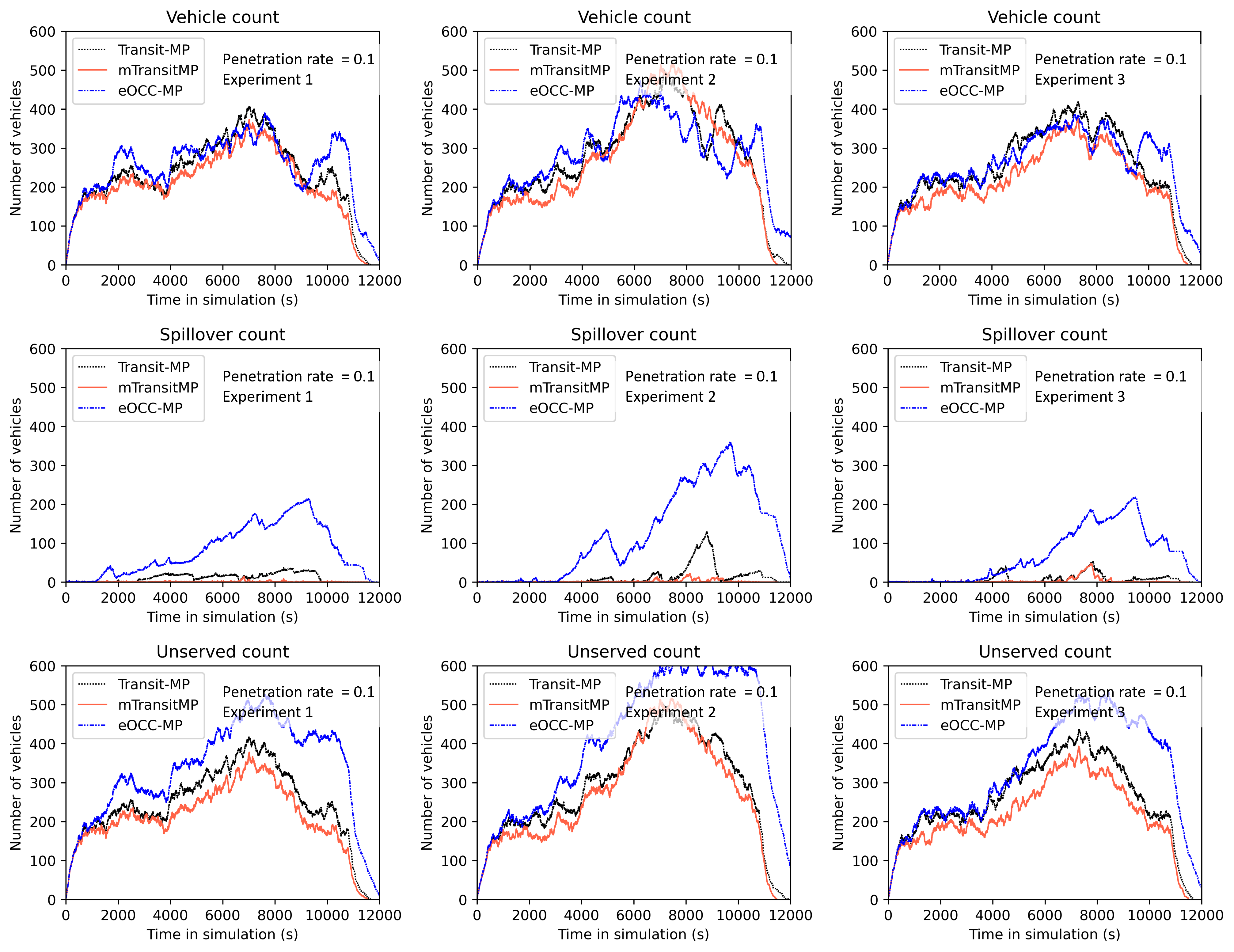}
    \caption{Real-time performance of mTransit-MP in sparse CV environments.}
    \label{fig: 0.1 penetration}  
\end{figure}

A closer look at spillover behavior explains this phenomenon. eOCC-MP begins accumulating spillover vehicles almost from the early stage of simulation, which is a clear sign of the queue starvation phenomenon. Because in our simulation, any vehicle waiting more than 1000~s will be teleported downstream to prevent permanent gridlock, the spillover backlog finally clears once demand on competing links subsides; hence the sudden increase in vehicle count after 9000~s. Transit-MP also experiences spillovers, but only intermittently. In Experiment 2, where demand is higher, these spillover episodes last longer. By contrast, mTransit-MP, which is augmented with historical traffic data for modification, greatly suppresses both the frequency and duration of spillovers, demonstrating the benefit of its modification mechanism in CV data-sparse conditions.

The unserved vehicle counts provide a more intuitive view of the effectiveness of each controller in maximizing throughput in sparse CV environments. Consistent with vehicle count and spillover count, mTransit-MP is the most efficient, Transit-MP follows, and eOCC-MP remains a distant third.

\subsection{Impact of parameter estimation errors}

\begin{table}[t]
  \centering
  \caption{Performance of mTransit-MP with different levels of parameter errors at 0.1 CV penetration rates}
  \label{tab: error impact}
  \begin{tabularx}{18cm}{c c *{5}{>{\raggedleft\arraybackslash}X}}
    \toprule
    \textbf{Method} & \textbf{Error level} & \textbf{Vehicle delay (s)} & \textbf{Passenger delay (s)} & \textbf{Max.\ spillover (veh)} & \textbf{Max.\ unserved (veh)} \\
    \midrule
    \multirow{7}{*}{mTransit-MP}
        &-50\%	&285.8 	&31.0 	&8.4 	&383.2 \\
        &-30\%	&281.2 	&30.0 	&12.6 	&380.4 \\
        &-20\%	&281.6 	&30.1 	&10.8 	&373.6 \\
        &-10\%	&281.1 	&29.9 	&10.0 	&381.8 \\
        &0	&281.0 	&30.2 	&10.4 	&372.4 \\
        &10\%	&279.6 	&29.9 	&10.0 	&372.8 \\
        &20\%	&284.6 	&30.3 	&10.4 	&373.0 \\
        &30\%	&291.0 	&31.1 	&11.6 	&383.0 \\
        &50\%	&285.0 	&29.9 	&12.0 	&380.0 \\
        &STD	&3.6 	&0.5 	&1.3 	&4.7 \\
    \midrule
    Transit-MP & -- & 336.7	& 33.0 & 37 & 417 \\
    eOCC-MP    & -- & 384.4 & 34.8& 214 & 527 \\
    \bottomrule
  \end{tabularx}
\end{table}

Recall that we use the accurate parameters, i.e., arrival rate $\hat{\lambda}_{i,o}$ and queue length $Q_{i,o}$, for each decision step of mTransit-MP in our previous sections to test its performance in the ideal cases. In this section, we manually add different estimation errors to the arrival rate and the queue length to test the performance of mTransit-MP in more realistic scenarios. The base scenario is experiment 1 at a 0.1 penetration rate. 

This section sets 9 error levels for both the arrival rate and the queue length, ranging from -50\% to 50\% with a step of 10\%, where negative errors indicate the underestimation and positive errors indicate overestimation. Each error level represents the expectation of the error, and the actual error for each parameter will be randomly generated within a range of 5\% up or down. For example, at -20\%, random errors are generated between [-25\%, -15\%]. Repetition experiments are conducted five times with different random seeds. 

Table \ref{tab: error impact} lists the detailed results of mTransit-MP for different error levels. As the absolute error increases, the overall performance of mTransit-MP decreases slightly, mainly in the impact on background traffic. When the absolute error is larger, the vehicle delay and the maximum spillover count are increased, but the passenger delay is almost unchanged. Comparing the results of over-estimation and under-estimation, it can be seen that the increase in vehicle delay and maximum spillover counts is greater in the over-estimation scenario than in the under-estimation scenario. However, mTransit-MP still significantly outperforms Transit-MP and eOCC-MP even at high absolute error levels, which again proves the necessity and effectiveness of our modification mechanism and that the modification mechanism is relatively robust, i.e., effective even with significant errors. The standard deviation (STD) results across different error levels show that the proposed mTransit-MP is insensitive to parameter errors, which demonstrates its potential for easy practical applications.

\section{Conclusion and Future Work}
On the basis of CV-MP \citep{tan2025cvmp}, we proposed Transit-MP for transit signal priority in partially CV environments. To prioritize high-occupancy transit vehicles, the real-time occupancy information of CVs at upstream links is incorporated into the pressure calculation. Besides, the impact of transit stations is also considered by counting transit vehicles only after they leave the nearest station. The incorporation of upstream occupancy information leads to a distinction of traffic state calculation for upstream and downstream links, which is the source of our contribution to the stability proof, more generally, and inclusive of existing MP controllers. We prove that Transit-MP can guarantee network stability even with partially CVs. In addition, considering the queue starvation phenomenon in sparse CV environments, i.e., a movement may no longer receive the green phase despite the queue spillover due to the lack of CV observations, we proposed modified Transit-MP, i.e., mTransit-MP, that incorporates historical traffic information, which has been shown to theoretically avoid this phenomenon that will lead to network destabilization. 

Through comprehensive simulation experiments at a real-world multi-modal corridor in Amsterdam with multiple tram and bus lines, we found that: 1) considering the impact of transit stations can further improve both the background and transit traffic performance, as evidenced by the improved performance (17.0\% reduction in multi-modal vehicle delays and 31.6\% reduction in passenger delays) of the extended OCC-MP (i.e., eOCC-MP) with considering transit stations compared with original OCC-MP without such a consideration \citep{ahmed2024occ}; 2) Transit-MP significantly outperforms eOCC-MP in various penetration rates with reduced transit and private vehicle delay as well as fewer real-time vehicle count and spillover count; at 0.2 penetration rates, the multi-modal vehicle delay reduced 21.9\% and the maximum spillover vehicles reduced 94.3\%; this is attributed to transit-MP taking into account the cumulative delay of vehicles; 3) mTransit-MP can further reduce spillover vehicles compared with Transit-MP by incorporating historical traffic information for a modification mechanism, especially in sparse CV environments; at a 0.1 penetration rate, the spillover vehicle reduced 61.8\% along with 14.2\% reduction in multi-modal vehicle delay and 11.7\% reduction in passenger delay; 4) the estimation error of historical traffic information has negligible impacts on mTransit-MP, demonstrating its potential for easy practical applications without careful and laborious calibrations. 5) An appropriate link segmentation strategy can enhance the performance of MP controllers (See Appendix \ref{appendix: seg results}). 

In the future, more traffic modes, including bicycles and pedestrians \citep{liu2024max, xu2022integrating} need to be considered together, which pose extra challenges on the design of pressure calculation as well as the stability of the network. Besides, the utilization of vehicular data for pressure calculation in CV-based MP controllers may impose privacy issues; thus, future work needs to integrate privacy-preserving mechanisms with MP controllers to protect the private data of CVs \citep{tan2024privacy, tan2025privacy}. Finally, although our sensitivity tests indicate that appropriately segmenting links can improve the performance of MP controllers, the optimal segmentation strategies differ across various MP controllers. How to segment long links remains worthy of further exploration in both theoretical and experimental contexts for MP control.

\section*{Acknowledgment}
Chaopeng Tan, Marco Rinaldi, and Hans van Lint would like to acknowledge the financial contribution of the EU Horizon Europe Research and Innovation Programme, Grant Agreement No. 101103808 ACUMEN; Hao Liu was supported by the NSF, United States Grant CMMI-2501601.

\appendix
\section{Determination of \texorpdfstring{$\beta_v$}{beta	extunderscore v} in more intelligent scenarios} \label{appendix: beta}
In more intelligent scenarios, if the remaining dwell time of the transit vehicle is known, then we can determine $\beta_v$ in a more precise manner. At each signal decision step, we let
\begin{align}
    \beta_v = \begin{cases}
        1, \text{ if } \alpha_v = 0, \\
        1, \text{ if } \alpha_v = 1, \text{ and } t_v^{EAT} < T_0, \\
        0, \text{ if } \alpha_v = 1, \text{ and } t_v^{EAT} \geq T_0.
    \end{cases} \label{eq: beta, precise}
\end{align}
where $T_0$ denotes the signal decision step length for MP control, and $t_v^{EAT}$ is the expected arrival time of the transit vehicle $v$. Note that there are many existing studies focused on the estimation of $t_v^{EAT}$ \citep{tan2008prediction, cvijovic2022conditional}. For the sake of methodological integrity, here we provide a simple estimate:
\begin{align}
    t_v^{EAT} = \begin{cases}
        \frac{ L_{ts}-x_v }{\varphi^{max}} + T_v^{dwell} + \frac{L- L_{ts}}{\varphi^{max}} + \theta, \text{ if } x_v \leq L_{ts} \\
        \frac{ L-x_v}{\varphi^{max}} + \theta , \text{ if } x_v > L_{ts}
    \end{cases}.
\end{align}
where $\varphi^{max}$ denotes the maximum speed; $T_v^{dwell}$ denotes the remaining dwell time, which is assumed to be known if the system is intelligent enough; $\theta$ is a buffer value to compensate for the acceleration and deceleration process. In the case of $x_v \leq L_{ts}$, i.e., the transit vehicle is on the upstream of the station or dwelling in the station, $\frac{ x_v - L_{ts}}{\varphi^{max}}$ and $\frac{L- L_{ts}}{\varphi^{max}}$ calculates the period arriving the station and the period leaving for stopline, respectively. In more complex scenarios with mixed private and public traffic, some methods also take into account the impact of residual queues for $t_v^{EAT}$ estimation \citep{liang2023two}. Since this is not the focus of our method, we will not discuss it further here.

\section{Stability of Transit-MP in heterogeneously distributed CV environments} \label{appendix: heterogeneously}
\begin{theorem}[Stability of Transit-MP in heterogeneously distributed CV environments] \label{thm: stability of Transit-MP in heterogeneously}
    Given the reduced admissible demand region $\bm{\Lambda}' =\{\bm{a} \mid \bm{a} \preceq \frac{\xi^{min}}{\xi^{max}}(\bm{I} - \bm{r})\bm{c}\bar{\bm{s}}, \quad \exists \bar{\bm{s}} \in \bm{S}^{co}\}$ and $\bm{\Lambda}' \subseteq \bm{\Lambda}$, if the exogenous demand $\bm{a} \in \bm{\Lambda}^{'int}$ (the interior of $\bm{\Lambda}'$), the proposed Transit-MP presented in Eq.~\eqref{eq: transit MP} (or Eq.~\eqref{eq: Transit-MP, density} in density form) can strongly stabilize the traffic network queues in heterogeneously distributed CV environments.  
\end{theorem}

\begin{proof}
    The first part for proving Theorem \ref{thm: stability of Transit-MP in heterogeneously} is identical to Theorem \ref{thm: stability of transit-MP}, diverging after Eq.~\eqref{eq: rewrite pressure of transit-mp}.

     According to \citet{tan2025cvmp}, if we assume that the probability of a vehicle being a CV follows a Bernoulli distribution, i.e., $Pr(v \in \mathcal{V}^{cv}_{i,o}) = \pi_{i,o}>0$, where $\pi_{i,o}$ can vary across different movements, then we have 
    \begin{align} 
        \pi^{min} \mathbb{E}^{\bm{\rho}(t)}(\bm{w}^d) \preceq \mathbb{E}^{\bm{\rho}(t)}(\bm{w}^{d,cv}) \preceq \pi^{max} \mathbb{E}^{\bm{\rho}(t)}(\bm{w}^d) 
    \end{align}
    where $\pi^{min}$ and $\pi^{max}$ are two constants determined by the penetration rates $\pi_{(i,o)}$ with $(i,o) \in \mathcal{M_{N \cup F}}$ across the network.

    Based on Eq.~\eqref{eq: positive pressure}, we have the original pressure for each movement is non-negative, i.e., $\{\{\bm{w}^{d}\}^\top (\bm{I}- \bm{r})\bm{c}\bm{s}^*\}_{\forall(i,o) \in \mathcal{M_N}} \geq 0$, then we have
    \begin{align}
        \mathbb{E}^{\bm{\rho}(t)} (\{\bm{w}^{d}\}^\top (\bm{I}- \bm{r})\bm{c}\bm{s}^*) & \geq \frac{1}{\pi^{max}} \mathbb{E}^{\bm{\rho}(t)} (\{\bm{w}^{d,cv}\}^\top (\bm{I}- \bm{r})\bm{c}\bm{s}^*)\nonumber \\
        &= \frac{1}{\pi^{max}} \mathbb{E}^{\bm{\rho}(t)}((\{\bm{w}^{u,cv}\}^\top - \{\bm{w}^{d,cv}\}^\top \bm{r})\bm{c}\bm{s}^* - (\{\bm{w}^{u,cv}\}^\top - \{\bm{w}^{d,cv}\}^\top)\bm{c}\bm{s}^*) \nonumber \\
        &\geq \frac{1}{\pi^{max}} \mathbb{E}^{\bm{\rho}(t)}(\{\bm{w}^{d,cv}\}^\top (\bm{I}- \bm{r})\bm{c}\bar{\bm{s}} - \underbrace{(\{\bm{w}^{u,cv}\}^\top - \{\bm{w}^{d,cv}\}^\top)\bm{c}(\bm{s}^*-\bar{\bm{s}})}_{\chi_1}).
    \end{align}
    
    Back to $\eta_1$, according to the definition of the reduced admissible region, there exists a positive $\epsilon$ that makes $\bm{a} -  \frac{\pi^{min}}{\pi^{max}}(\bm{I} - \bm{r})\bm{c}\bar{\bm{s}} \preceq -\epsilon \bm{1}$, then we have
    \begin{align}
         \eta_1 =& \mathbb{E}^{\bm{\rho}(t)}\left(\{\bm{w}^d\}^\top \bm{a} - \{\bm{w}^d\}^\top(\bm{I}-\bm{r}) \bm{c}\bm{s}^*\right) \nonumber \\
         \leq & \frac{1}{\pi^{min}} \mathbb{E}^{\bm{\rho}(t)}(\{\bm{w}^{d,cv}\}^\top \bm{a} - \pi^{min}\{\bm{w}^d\}^\top(\bm{I}-\bm{r}) \bm{c}\bm{s}^* ) \nonumber \\
         \leq & \frac{1}{\pi^{min}} \mathbb{E}^{\bm{\rho}(t)}(\{\bm{w}^{d,cv}\}^\top \bm{a} - \frac{\pi^{min}}{\pi^{max}}\{\bm{w}^{d,cv}\}^\top(\bm{I}-\bm{r}) \bm{c}\bar{\bm{s}} ) + \frac{1}{\pi^{max}} \mathbb{E}^{\bm{\rho}(t)} (\chi_1) \nonumber \\
         \leq & \frac{1}{\pi^{min}} \mathbb{E}^{\bm{\rho}(t)}\left(\{\bm{w}^{d,cv}\}^\top ( \bm{a} - \frac{\pi^{min}}{\pi^{max}} (\bm{I}-\bm{r}) \bm{c}\bar{\bm{s}} )\right) + \frac{1}{\pi^{max}} \mathbb{E}^{\bm{\rho}(t)}(\chi_1) \nonumber \\
         \leq & -\frac{1}{\pi^{min}} \epsilon \mathbb{E}^{\bm{\rho}(t)}(\{\bm{w}^{d,cv}\}^\top \bm{1}) + \frac{1}{\pi^{max}}\mathbb{E}^{\bm{\rho}(t)}(\chi_1) \nonumber \\
         \leq & -\epsilon \mathbb{E}^{\bm{\rho}(t)}(\{\bm{w}^d\}^\top \bm{1}) + \frac{1}{\pi^{max}}\mathbb{E}^{\bm{\rho}(t)}(\chi_1)
    \end{align}

    Note that, if the same weights are used for upstream and downstream movement states for pressure calculation, like Q-MP \citep{varaiya2013max}, D-MP \citep{liu2022novel}, and CV-MP \citep{tan2025cvmp}, $\chi_1 = 0$. This suggests that our stability proof is more generalized.
    Specifically, in our cases, 
    \begin{align}
        0\leq w^{u,cv}_{i,o}(t) - w^{d,cv}_{i,o}(t) 
        \begin{cases}
            \leq w^{u,cv}_{i,o}(t) \quad (i,o) \in \mathcal{M_N} \\
            =0 \quad (i,o) \in \mathcal{M_F} 
        \end{cases}
    \end{align}

    Obviously, for the second term of $\eta_1$ we have 
    \begin{align}
        \frac{1}{\pi^{max}}\mathbb{E}^{\bm{\rho}(t)}(\chi_1) &\leq \frac{1}{\pi^{max}}\mathbb{E}^{\bm{\rho}(t)}((\{\bm{w}^{u,cv}\}^\top - \{\bm{w}^{d,cv}\}^\top)\bm{c}\bm{s}^*) \leq \frac{1}{\pi^{max}}\mathbb{E}^{\bm{\rho}(t)}((\{\bm{w}^{u,cv}\}^\top - \{\bm{w}^{d,cv}\}^\top)\bm{c}) \nonumber \\
        &\leq \frac{1}{\pi^{max}}\sum_{\mathcal{M}_{\mathcal{N}}}\mathbb{E}^{\bm{\rho}(t)} (c\int_0^{L_i} \tau^u(x,t) \rho^{cv}(x,t) \mathrm{d}x)\leq \frac{1}{\pi^{max}}\sum_{\mathcal{M}_{\mathcal{N}}} c \tau^{u,max} L_i \rho^{max} \triangleq K_5.
    \end{align}

    The remainder of the proof, starting from Eq.~\eqref{eq: K5}, is identical to Theorem \ref{thm: stability of transit-MP}.
\end{proof}

\section{Impact of link segmentation} \label{appendix: seg results}
As we discussed in Section \ref{ssec: fully CV}, including all vehicles directly on the link in pressure calculations may cause premature inclusion of vehicles on long links, resulting in wasted green time, particularly when prioritizing vehicles with high occupancy. Therefore, this section will test the effects of several link segmentation strategies on MP controllers in fully CV environments. The segmentation strategies vary by the length of the links and the decision steps used. S0 uses the actual link length with no segmentation (default no segmentation), S1 segments links based on the shortest link length (about 90 m), while S2-S5 segment links based on decision steps, ranging from the travel distance of free-flow vehicles during one decision step (S2) to multiples of decision steps (S3-S5), with S2 at 140 m, S3 at 280 m, S4 at 420 m, and S5 at 560 m. Note that when segmenting the links, we only change the set of vehicles used to calculate pressure, without recalculating the state of the vehicles. The vehicle delay in this section is corrected by considering the time loss of spillover vehicles.

Fig. \ref{fig: seg results} presents the performance of three MP controllers, i.e., Transit-MP, OCC-MP, and eOCC-MP, under various link segmentation strategies. It can be observed that different MP controllers exhibit distinct preferences for link segmentation strategies. Transit-MP performs best at S3 (420m) and even surpasses no segmentation (S0, with actual links longer than 900 m). Without segmentation, all vehicles on a long link influence pressure, even those far from the stop line; Transit-MP, using travel time as the basis for pressure calculation does discount distant vehicles, but including extremely long upstream tails does not improve control. Conversely, overly short segments (S1, 90 m) fragment queues and reduce inter-movement travel-time contrast inside segments, causing the controller to misread relative pressures. S3 can effectively cover vehicle queues and exclude the impact of upstream tail vehicles most of the time, providing a more accurate reflection of road traffic conditions and thus delivering superior performance.
It is worth emphasizing that Transit-MP consistently avoids spillovers, underscoring its adaptability to various link lengths.

eOCC-MP outperforms OCC-MP in every scenario because it accounts for transit station effects in the pressure calculation. It should be noted that this advantage diminishes as link lengths are segmented shorter. This is because in scenarios with shorter links, the green lights wasted by OCC-MP due to transit priority can be reduced, thereby weakening the advantage gained from accounting for station impacts. For both counted-based controllers, OCC-MP and eOCC-MP, any segmentation strategies (S1–S5) beat S0 on delay, with shorter segments generally yielding lower measured delays but substantially more spillovers. Mechanistically, when links are segmented into short (S1-S2), queues often extend beyond the segmented link boundaries. OCC-MP and eOCC-MP, relying on vehicle counts, tend to allocate green time more evenly across movements, which is sub-optimal when long queues concentrate on specific movements. In our corridor, this manifests as spillovers, particularly at traffic generation point 5. In S3-S5 with longer segmented links, the vehicle counts within segmented links better reflect movement queues, thereby accelerating the dissipation of long queues and reducing spillovers. However, this comes at the cost of allowing more vehicles to enter the road network, which increases vehicle delays.

In summary, when comparing the optimal performance of each MP controller under various segmentation strategies, Transit-MP at S3 (420 m) still outperforms OCC-MP and eOCC-MP at S1/S2 (140 m/280 m), achieving lower vehicle delays and fewer spillover vehicles.

\begin{figure}[ht]
    \centering
    \includegraphics[width=0.9\textwidth]{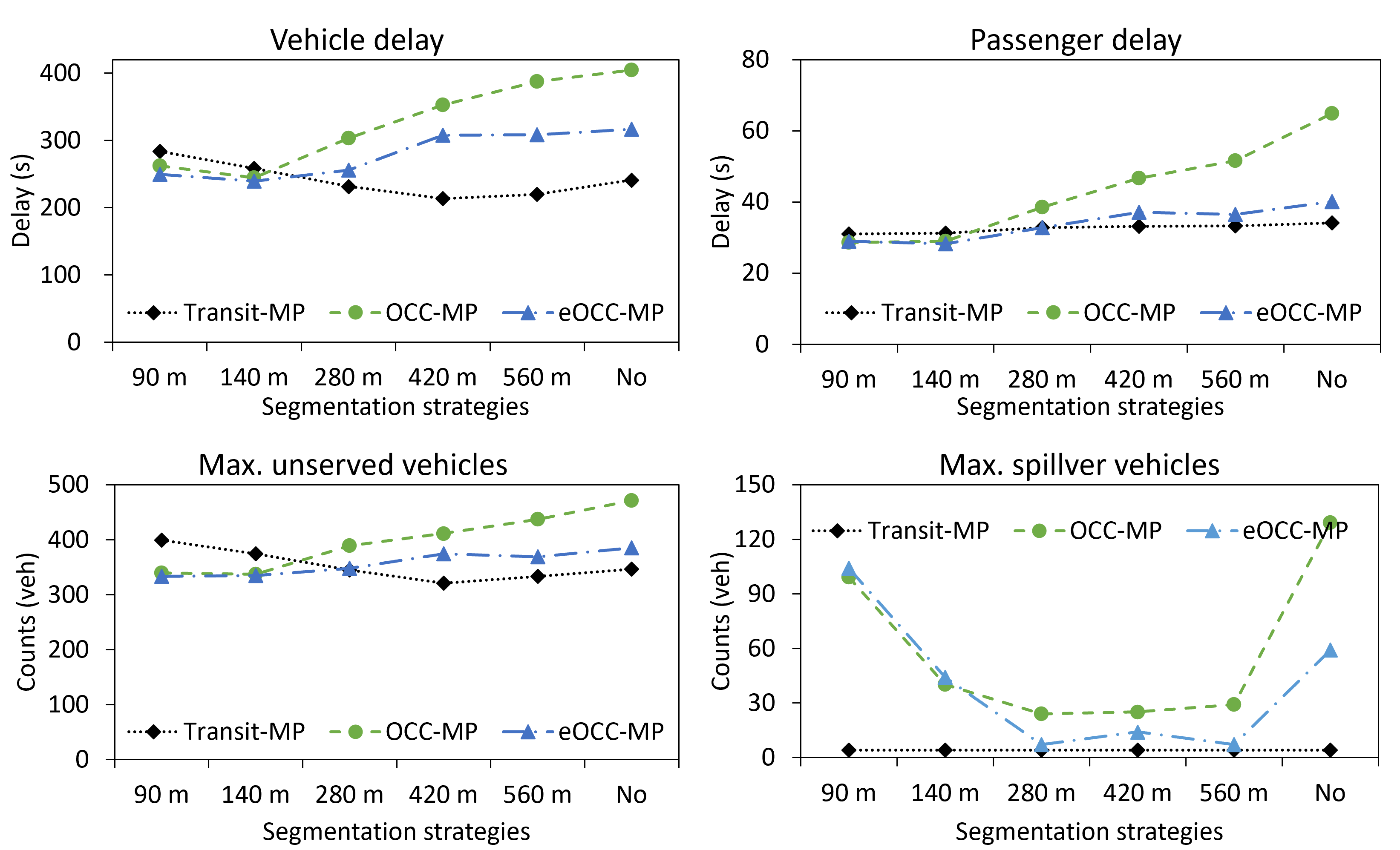}
    \caption{Performance of MP controllers under various link segmentation strategies.}
    \label{fig: seg results}  
\end{figure}

\bibliographystyle{apalike} 
\bibliography{Transit-MP_TR-B}
\end{document}